\documentclass[a4paper,11pt]{amsart}

\usepackage{a4wide}

\usepackage[breaklinks,bookmarks=false]{hyperref}
\hypersetup{colorlinks,linkcolor=blue,citecolor=blue,urlcolor=blue,plainpages=false,pdfwindowui=false}

\usepackage{pgfplots}
\usepackage{accents}
\usepackage{subcaption}
\usepackage{subfloat}
\usepackage[foot]{amsaddr}

\usepackage{amssymb}
\usepackage{mathrsfs}

\newcommand{\eq}{:=}

\newcommand{\grad}{\boldsymbol \nabla}
\renewcommand{\div}{\grad \cdot}
\newcommand{\curl}{\grad \times}

\newcommand{\ccurl}{\boldsymbol{\operatorname{curl}}}
\newcommand{\ddiv}{\operatorname{div}}

\newcommand{\BD}{\boldsymbol D}
\newcommand{\BE}{\boldsymbol E}

\newcommand{\BG}{\boldsymbol G}
\newcommand{\BH}{\boldsymbol H}
\newcommand{\BI}{\boldsymbol I}
\newcommand{\BJ}{\boldsymbol J}

\newcommand{\BL}{\boldsymbol L}

\newcommand{\BW}{\boldsymbol W}

\newcommand{\ba}{\boldsymbol a}
\newcommand{\bb}{\boldsymbol b}

\newcommand{\be}{\boldsymbol e}
\newcommand{\bg}{\boldsymbol g}

\newcommand{\bn}{\boldsymbol n}

\newcommand{\br}{\boldsymbol r}

\newcommand{\bv}{\boldsymbol v}
\newcommand{\bw}{\boldsymbol w}
\newcommand{\bx}{\boldsymbol x}

\newcommand{\CE}{\mathcal E}

\newcommand{\CI}{\mathcal I}

\newcommand{\CP}{\mathcal P}

\newcommand{\CT}{\mathcal T}

\newcommand{\CV}{\mathcal V}

\newcommand{\LC}{\mathscr C}

\newcommand{\LP}{\mathscr P}

\newcommand{\BCP}{\boldsymbol{\CP}}

\newcommand{\LBA}{\boldsymbol \Lambda}

\newcommand{\contrasteps}[1]{\LC_{\varepsilon,#1}}
\newcommand{\contrastmu }[1]{\LC_{\mu,#1}}

\newcommand{\ess}{\operatorname{ess}}

\newcommand{\oma}{\omega^{\ba}}
\newcommand{\omb}{\omega^{\bb}}

\newcommand{\omK}{\omega_K}
\newcommand{\tomK}{\widetilde{\omega}_K}

\newcommand{\psia}{\psi^{\ba}}

\newcommand{\CTa}{\CT_h^{\ba}}
\newcommand{\CTK}{\CT_{h,K}}

\newcommand{\tCTa}{\widetilde{\CT}_h^{\ba}}
\newcommand{\tCTK}{\widetilde{\CT}_{h,K}}

\newcommand{\ha}{h_{\oma}}
\newcommand{\cmina}{\cminD{\oma}}
\newcommand{\hb}{h_{\omb}}
\newcommand{\cminb}{\cminD{\omb}}

\newcommand{\epsmax}{\varepsilon_{\max}}
\newcommand{\epsmin}{\varepsilon_{\min}}

\newcommand{\mumaxD }[1]{\mu_{\max,#1}}
\newcommand{\muminD }[1]{\mu_{\min,#1}}
\newcommand{\chimaxD}[1]{\chi_{\max,#1}}
\newcommand{\chiminD}[1]{\chi_{\min,#1}}
\newcommand{\epsmaxD}[1]{\varepsilon_{\max,#1}}
\newcommand{\epsminD}[1]{\varepsilon_{\min,#1}}

\newcommand{\cminD}[1]{c_{\min,#1}}

\newcommand{\chimaxa}{\chimaxD{\oma}}
\newcommand{\chimina}{\chiminD{\oma}}
\newcommand{\epsmaxa}{\epsmaxD{\oma}}
\newcommand{\epsmina}{\epsminD{\oma}}

\newcommand{\toma}{\widetilde{\omega}^{\ba}}

\newcommand{\RT}{\boldsymbol{RT}}
\newcommand{\ND}{\boldsymbol{N}}

\newcommand{\bzero}{\boldsymbol 0}

\newcommand{\ee}{\boldsymbol \varepsilon}
\newcommand{\cc}{\boldsymbol \chi}
\newcommand{\mm}{\boldsymbol \mu}

\newcommand{\btheta}{\boldsymbol \theta}
\newcommand{\tbtheta}{\widetilde{\btheta}}
\newcommand{\hbtheta}{\widehat{\btheta}}

\newcommand{\bxi}{\boldsymbol \xi}
\newcommand{\bphi}{\boldsymbol \phi}

\newcommand{\enorm}[1]{\left |\!\left |\!\left |#1\right |\!\right |\!\right |}

\newcommand{\gba}{\gamma_{\rm ba}}
\newcommand{\gst}{\gamma_{\rm st}}

\newcommand{\btau}{\boldsymbol \tau}

\newcommand{\pa}{p^{\ba}}
\newcommand{\pK}{p_K}

\newcommand{\SlopeTriangle}[6]
{

    \pgfplotsextra
    {
        \pgfkeysgetvalue{/pgfplots/xmin}{\xmin}
        \pgfkeysgetvalue{/pgfplots/xmax}{\xmax}
        \pgfkeysgetvalue{/pgfplots/ymin}{\ymin}
        \pgfkeysgetvalue{/pgfplots/ymax}{\ymax}

        \pgfmathsetmacro{\xArel}{#1}
        \pgfmathsetmacro{\yArel}{#3}
        \pgfmathsetmacro{\xBrel}{#1-#2}
        \pgfmathsetmacro{\yBrel}{\yArel}
        \pgfmathsetmacro{\xCrel}{\xArel}

        \pgfmathsetmacro{\lnxB}{\xmin*(1-(#1-#2))+\xmax*(#1-#2)} 
        \pgfmathsetmacro{\lnxA}{\xmin*(1-#1)+\xmax*#1} 
        \pgfmathsetmacro{\lnyA}{\ymin*(1-#3)+\ymax*#3} 
        \pgfmathsetmacro{\lnyC}{\lnyA+#4*(\lnxA-\lnxB)}
        \pgfmathsetmacro{\yCrel}{\lnyC-\ymin)/(\ymax-\ymin)} 

        \coordinate (A) at (rel axis cs:\xArel,\yArel);
        \coordinate (B) at (rel axis cs:\xBrel,\yBrel);
        \coordinate (C) at (rel axis cs:\xCrel,\yCrel);

        \draw[#6]   (A)--
                    (B)--
                    (C)-- node[anchor=east] {#5}
                    cycle;
    }
}

\newtheorem{theorem}{Theorem}
\newtheorem{lemma}[theorem]{Lemma}

\newtheorem{remark}[theorem]{Remark}

\numberwithin{equation}{section}
\numberwithin{theorem}{section}
\numberwithin{figure}{section}

\newcommand{\figsize}{.40}

\title
[
A posteriori estimates for Maxwell's equations%
]
{
Asymptotically constant-free and polynomial-degree-robust a posteriori
error estimates for time-harmonic Maxwell's equations
}

\author{T. Chaumont-Frelet$^\star$}

\address{\vspace{-.5cm}}
\address{\noindent \tiny \textup{$^\star$Inria, Univ. Lille, CNRS, UMR 8524 -- Laboratoire Paul Painlev\'e}}

\begin{document}

\thispagestyle{empty}

\maketitle

\begin{abstract}
We propose a novel a posteriori error estimator for the N\'ed\'elec finite element
discretization of time-harmonic Maxwell's equations. After the approximation of the 
electric field is computed, we propose a fully localized algorithm to reconstruct
approximations to the electric displacement and the magnetic field, with such
approximations respectively fulfilling suitable divergence and curl constraints.
These reconstructed fields are in turn used to construct an a posteriori error estimator
which is shown to be reliable and efficient. Specifically, the estimator controls the
error from above up to a constant that tends to one as the mesh is refined and/or the
polynomial degree is increased, and from below up to constant independent of $p$. Both
bounds are also fully-robust in the low-frequency regime. The properties of the proposed
estimator are illustrated on a set of numerical examples.

\vspace{.5cm}

\end{abstract}

\section{Introduction}

We consider time-harmonic Maxwell's equations in an heterogeneous domain $\Omega$:
Given $\BJ: \Omega \to \mathbb C^3$, the solution $\BE: \Omega \to \mathbb C^3$
satisfies
\begin{equation}
\label{eq_maxwell_strong}
\left \{
\begin{array}{rcll}
-\omega^2 \ee \BE + \curl \left (\mm^{-1} \curl \BE\right ) &=& i\omega\BJ & \text{ in } \Omega,
\\
\BE \times \bn &=& \bzero & \text{ on } \partial \Omega,
\end{array}
\right .
\end{equation}
where $\omega > 0$ is the frequency, and $\ee$ and $\mm$ are symmetric tensor-valued
functions describing the electric permittivity and magnetic permeability of the materials
contained inside $\Omega$. Boundary value problem~\eqref{eq_maxwell_strong} is instrumental
in a number of applications involving electrodynamics~\cite{griffiths_1999a}, where $\BJ$
represents an applied current density, and $\BE$ is the (unknown) electric field.

Apart from simple geometrical configurations, the analytical solution to
\eqref{eq_maxwell_strong} is unavailable, which motivates the use of numerical
schemes to compute approximate solutions. Here, we consider conforming N\'ed\'elec
finite element discretizations~\cite{ciarlet_2002a,monk_2003a,nedelec_1980a,nedelec_1986a},
and focus on two key aspects: reliable error estimation and adaptive mesh refinements using
equilibrated error estimators.

The use of residual-based a posteriori error estimators has been considered in
the literature for both time-harmonic Maxwell's equations~\cite{chaumontfrelet_vega_2022a}
and the simpler model of the Helmholtz problem which is limited to specifically
polarized electromagnetic fields~\cite{dorfler_sauter_2013a,sauter_zech_2015a}.
A shortcoming of residual-based estimators, however, is that they only control
the error up to a generic constant depending on the flatness of mesh elements, even
for asymptotically fine meshes. Besides, the efficiency constant typically depends
on the polynomial degree $p$, meaning in particular that the upper bound is usually
not very tight for high-order elements.

Equilibrated error estimators are slightly more complicated to compute than residual-based
ones, but they alleviate the two aforementioned limitations~\cite{braess_pillwein_schoberl_2009a,%
ern_vohralik_2021a,prager_synge_1947a}. Unfortunately, although equilibrated estimators have
been developed and analyzed for the Helmholtz problem~\cite{chaumontfrelet_ern_vohralik_2021a,%
congreve_gedicke_perugia_2019a}, there is currently no construction for time-harmonic Maxwell's
equations available in the literature. More precisely, equilibrated estimators have been proposed
for the magnetostatic problem (which similar to~\eqref{eq_maxwell_strong}, but with $\omega = 0$)
recently~\cite{braess_schoberl_2008a,chaumontfrelet_vohralik_2023a,gedicke_geevers_perugia_2020a,%
gedicke_geevers_perugia_schoberl_2021a}, and the purpose of this work is to bridge these results
to the time-harmonic case.

As compared to the magnetostatic problem covered in~\cite{braess_schoberl_2008a,%
chaumontfrelet_vohralik_2023a,gedicke_geevers_perugia_2020a,gedicke_geevers_perugia_schoberl_2021a},
the time-harmonic case considered here presents two key difficulties. (i) It is not
coercive, and understanding the dependency of the reliability and efficiency constants
in terms of the frequency is crucial. This difficulty is already present for Helmholtz
problems, and it is addressed using duality techniques~\cite{chaumontfrelet_ern_vohralik_2021a,%
dorfler_sauter_2013a,sauter_zech_2015a}. (ii) Duality techniques are significantly more
complicated to perform for Maxwell's equations than for Helmholtz problems~\cite{%
chaumontfrelet_ern_2023a,ern_guermond_2018a,monk_1992a,zhong_shu_wittum_xu_2009a}.
Intuitively, the lack of compactness in the embedding $\BH_0(\ccurl,\Omega) \subset \BL^2(\Omega)$
requires to precisely take into account the divergence constraint
\begin{equation}
\label{eq_divergence_constraint}
-\omega^2 \div(\ee\BE) = i\omega \div \BJ
\end{equation}
hidden in \eqref{eq_maxwell_strong}. It is not an easy task, since N\'ed\'elec finite
elements are $\BH_0(\ccurl,\Omega)$-conforming, but not $\BH(\ddiv,\Omega)$-conforming.

In the context of residual-based error estimators, these difficulties have recently been
addressed in~\cite{chaumontfrelet_vega_2022a}. The key idea is to not only measure the
residual appearing in~\eqref{eq_maxwell_strong}, but also include a contribution that
accounts for the violation of~\eqref{eq_divergence_constraint}. With such an estimator,
the duality techniques developed in~\cite{chaumontfrelet_ern_2023a,ern_guermond_2018a,%
monk_1992a,zhong_shu_wittum_xu_2009a} in the context of a priori error analysis can be
bridged to produce a posteriori error estimates.

In this work, we propose an equilibrated error estimator following the ideas
of~\cite{chaumontfrelet_vega_2022a}. We therefore reconstruct two equilibrated
fields respectively meant to measure how~\eqref{eq_maxwell_strong}
and~\eqref{eq_divergence_constraint} are violated by the discrete solution $\BE_h$.
Specifically, these two fields physically correspond to an electric displacement $\BD_h$
and a magnetic field $\BH_h$ and satisfy the constraints
\begin{equation*}
-\omega^2\div \BD_h =  i\omega\div\BJ,
\qquad
i\omega \curl \BH_h = i\omega\BJ+\omega^2 \BD_h.
\end{equation*}
The construction of the divergence-constrained electric displacement $\BD_h$
closely follows the equilibration techniques introduced for the Poisson
problem~\cite{braess_pillwein_schoberl_2009a,destuynder_metivet_1999a,ern_vohralik_2015a},
whereas we adapt the construction in~\cite{chaumontfrelet_vohralik_2023a} for the magetic
field. In both cases, these reconstructions only hinge on small uncoupled finite element
problems set on vertex patches which may be solved in parallel.

With the equilibrated fields $\BD_h$ and $\BH_h$ constructed,
the estimator is defined as
\begin{equation*}
\eta_K^2
\eq
\omega^2 \|\BE_h-\ee^{-1}\BD_h\|_{\ee,K}^2
+
\|\curl \BE_h-i\omega\mm\BH_h\|_{\mm^{-1},K}^2
\end{equation*}
for every mesh element $K$. Our key result is that
\begin{equation}
\label{eq_estimates_intro}
\enorm{\BE-\BE_h}_{\omega,\Omega}^2
\leq
(1+\theta_{\rm rel}^2) \sum_{K} \eta_K^2,
\qquad
\eta_K
\lesssim
(1+\theta_{\rm eff}) \enorm{\BE-\BE_h}_{\omega,\tomK},
\end{equation}
where
\begin{equation*}
\enorm{\BE-\BE_h}_{\omega,D}^2
\eq
\omega^2\|\BE-\BE_h\|_{\ee,D}^2 + \|\curl(\BE-\BE_h)\|_{\mm^{-1},D}^2,
\qquad
D \subset \Omega,
\end{equation*}
is the natural energy norm and the coefficient-weighted $\BL^2(\Omega)$-norms
are rigorously introduced in Section \ref{section_functional_spaces} below.
In~\eqref{eq_estimates_intro}, $\theta_{\rm rel}$ tends to zero when the mesh
is refined and/or the polynomial degree is increased, $\theta_{\rm eff}$ is
inversely proportional to the number of elements per wavelength,
$\tomK$ is a small region around the element $K$,
and the symbol $\lesssim$ means that the inequality holds up to
constant that only depends on the shape-regularity parameter of
the mesh. Crucially, our upper bound is asymptotically constant-free,
and the lower bound does not depend on the polynomial degree~$p$.
Besides, the dependency of $\theta_{\rm rel}$ on the mesh size,
the polynomial degree, and the frequency is exactly the same as in
the simpler case of Helmholtz problems. Hence, our results are in
a certain sense optimal. We also emphasize that we do not assume that
$\Omega$ is simply-connected, and that the estimates in~\eqref{eq_estimates_intro}
hold true irrespectively of the topology of the domain. Finally, our bounds
are fully-robust, irrespectively of the mesh size and polynomial degree,
in the low-frequency regime where $\omega \to 0$.

We also present a set of numerical examples illustrating how the
constants in the lower bound and upper bound depend on the key
parameters of the model and discretization. In particular, they
show that the estimator indeed provides a very sharp constant-free
upper bound for sufficiently fine meshes.

The remainder of this work is organized as follows. In Section~\ref{section_setting},
we recap notation and collect key results from the literature. Our main findings concerning
the construction of the estimator and the corresponding error bounds are stated in
Section~\ref{section_main_results}, and later proved in Sections~\ref{section_reliability}
and~\ref{section_efficiency}. Finally, numerical examples are presented in
Section~\ref{section_numerical_examples}.

\section{Settings}
\label{section_setting}

This section introduces key notation and preliminary results.

\subsection{Domain and coefficients}

In this work, $\Omega \subset \mathbb R^3$ is a Lipschitz polyhedral domain.

In order to consider piecewise constant coefficients, we assume that $\Omega$
is subdivided into a ``physical partition'' $\LP$ consisting of a finite number
of (open) polyhedral Lipschitz subdomains $P$ such that $P \cap Q = \emptyset$
for all $P,Q \in \LP$ and $\overline{\Omega} = \cup_{P \in \LP} \overline{P}$.

We consider two real symmetric tensor-valued functions $\ee,\cc: \Omega \to \mathbb S$.
$\ee$ and $\cc$ are piecewise constant onto the partition $\LP$ which means that for
all $P \in \LP$, there exists two symmetric tensors $\ee_P,\cc_P \in \mathbb S$
such that $\ee(\bx) = \ee_P$ and $\cc(\bx) = \cc_P$ for all $\bx \in P$. We
further assume that $\ee$ and $\cc$ are uniformly elliptic, which means that
\begin{equation*}
\min_{P \in \LP} \min_{\bxi \in \mathbb R^3} \ee_P \bxi \cdot \bxi > 0
\qquad
\min_{P \in \LP} \min_{\bxi \in \mathbb R^3} \cc_P \bxi \cdot \bxi > 0.
\end{equation*}
If $\bx \in \Omega$, the notations $\epsmin(\bx)$ and $\epsmax(\bx)$ stand for the minimal
and maximal eigenvalues of $\ee(\bx)$, and if $D \subset \Omega$ is a measurable set, we introduce
\begin{equation*}
\epsminD{D} \eq \ess \inf_D \epsmin,
\qquad
\epsmaxD{D} \eq \ess \sup_D \epsmin.
\end{equation*}
We also employ similar notations for $\cc$ and $\mm \eq \cc^{-1}$.

\subsection{Contrast and wavespeed}

If $D \subset \Omega$ is an measurable set, we introduce the coefficient contrasts
\begin{equation*}
\contrasteps{D} \eq \frac{\epsmaxD{D}}{\epsminD{D}},
\qquad
\contrastmu{D}  \eq \frac{\mumaxD{D} }{\muminD {D}}.
\end{equation*}
We also define the minimal wavespeed over $D$ as
\begin{equation*}
\cminD{D}
\eq
\frac{1}{\sqrt{\varepsilon_{\max,D}\mu_{\max,D}}}.
\end{equation*}

\subsection{Functional spaces}
\label{section_functional_spaces}

If $D \subset \Omega$ is an open subset with Lipschitz boundary,
$L^2(D)$ stands for the Lebesgue space of real-valued square integrable functions
defined on $D$, and $\BL^2(D) \eq [L^2(D)]^d$ contains vector-valued
functions~\cite{adams_fournier_2003a}. The notations $(\cdot,\cdot)_D$
and $\|{\cdot}\|_D$ are respectively used for the inner-products and norms
of both $L^2(D)$ and $\BL^2(D)$. Besides, if $\boldsymbol \rho: D \to \mathbb S$
is a measurable symmetric matrix function whose eigenvalues are uniformly bounded
away from $0$ and $\infty$, then the application
\begin{equation*}
\|\bv\|_{\boldsymbol \rho,D}^2 \eq \int_D \boldsymbol \rho \bv \cdot \bv,
\qquad
\bv \in \BL^2(D)
\end{equation*}
is a norm on $\BL^2(D)$ equivalent to its natural norm.

The symbols $\grad$, $\curl$ and $\div$ stand for the weak gradient, curl and
divergence operators defined in the sense of distributions, and
$H^1(D) \eq \{ v \in L^2(D); \; \grad v \in \BL^2(D) \}$,
$\BH(\ccurl,D) \eq \{ \bv \in \BL^2(D); \; \curl \bv \in \BL^2(D) \}$
and
$\BH(\ddiv,D) \eq \{ \bv \in \BL^2(D); \; \div \bv \in L^2(D) \}$
are the associated Sobolev spaces.
We refer the reader to~\cite{adams_fournier_2003a,girault_raviart_1986a}
for an in-depth presentation. The space
\begin{equation*}
\BH(\ddiv,\ee,D) \eq \left \{
\bv \in \BL^2(D) \; | \; \div(\ee\bv) \in L^2(D)
\right \}
\end{equation*}
will also be useful, and $\BH(\ddiv^0,D)$ and $\BH(\ddiv^0,\ee,D)$
respectively collect the elements $\bv$ of $\BH(\ddiv,D)$ and $\BH(\ddiv,\ee,D)$
such that $\div \bv = 0$ and $\div(\ee\bv) = 0$.

Finally, $\BH_0(\ccurl,\Omega)$ denotes the closure of smooth compactly supported
functions into $\BH(\ccurl,\Omega)$, i.e. elements of $\BH(\ccurl,\Omega)$ with
vanishing tangential trace.

\subsection{Energy norm}

If $D \subset \Omega$ is an open set with Lipschitz boundary, the application
\begin{equation*}
\enorm{\bv}_{\omega,D}^2 \eq \omega^2 \|\bv\|_{\ee,D}^2 + \|\curl \bv\|_{\cc,D}^2,
\quad
\bv \in \BH(\ccurl,D),
\end{equation*}
is a norm on $\BH(\ccurl,D)$ which is natural to analyze the problem.

\subsection{Well-posedness}

Assuming that $\BJ \in \BL^2(\Omega)$, we recast
Maxwell's equations~\eqref{eq_maxwell_strong} into a weak form where
we look for $\BE \in \BH_0(\ccurl,\Omega)$ such that
\begin{equation}
\label{eq_maxwell_weak}
b(\BE,\bv) = i\omega(\BJ_h,\bv)_\Omega \quad \forall \bv \in \BH_0(\ccurl,\Omega)
\end{equation}
with
\begin{equation*}
b(\BE,\bv) \eq -\omega^2 (\ee\BE,\bv)_\Omega + (\cc\curl \BE,\curl \bv)_\Omega.
\end{equation*}

We will work under the assumption that $b$ is inf-sup
stable, which means that there exists a constant $\gst > 0$
such that
\begin{equation}
\label{eq_inf_sup}
\inf_{\substack{\be \in \BH_0(\ccurl,\Omega) \\ \enorm{\be}_{\omega,\Omega} = 1}}
\sup_{\substack{\bv \in \BH_0(\ccurl,\Omega) \\ \enorm{\bv}_{\omega,\Omega} = 1}}
b(\be,\bv)
\geq
\frac{1}{\gst}.
\end{equation}
Notice that~\eqref{eq_inf_sup} guaranties the well-posedness
of~\eqref{eq_maxwell_weak}. The dependence of $\gst$ with respect to
$\omega$ is studied, e.g., in~\cite{chaumontfrelet_vega_2022b}.

\subsection{Computational mesh}

We consider a mesh $\CT_h$ that partitions $\Omega$ into a collection of
(open) tetrahedral elements $K$. We assume that the mesh is conforming in
the sense that the intersection $\overline{K}_+ \cap \overline{K}_-$ of
distinct elements $K_\pm \in \CT_h$ is either empty, or a full face, edge
or vertex of both $K_+$ and $K_-$. These assumption are standard~\cite{ciarlet_2002a,monk_2003a}
and do not prevent strong mesh grading.

For each $K \in \CT_h$, $h_K$ and $\rho_K$ respectively denote the diameter
of the smallest ball containing $K$ and of the largest ball contained in $\overline{K}$,
and $\CV(K)$ is the set of vertices of $K$.
$\kappa_K \eq h_K/\rho_K$ then stands for the shape-regularity parameter of
$K$, and if $\CT \subset \CT_h$, we let $\kappa_{\CT} \eq \max_{K \in \CT} \kappa_K$.
We also introduce $h \eq \max_{K \in \CT_h} h_K$.

\subsection{Hidden constants}

In order to lighten the notation in the remainder of this work,
we will write that $A \lesssim_{\CT} B$ if $A,B > 0$ are two
constants such that $A \leq C(\kappa_{\CT}) B$ where the
hidden constant $C(\kappa_{\CT})$ only depends on $\kappa_{\CT}$,
$\CT \subset \CT_h$ being a submesh.

\subsection{Finite element spaces}

If $q \geq 0$ and $K \in \CT_h$, we denote by $\CP_q(K)$
the space of polynomials of degree less than or equal to $q$,
and $\BCP_q(K) \eq [\CP_q(K)]^3$ contains vector-valued polynomial
functions. We will also employ the spaces $\ND_q(K) \eq \bx \times \BCP_q(K) + \BCP_q(K)$
and $\RT_q(K) \eq \bx \CP_q(K) + \BCP_q(K)$ of N\'ed\'elec and Raviart--Thomas polynomials
\cite{nedelec_1980a,raviart_thomas_1977a}.

If $\CT \subset \CT_h$ is a submesh, then $\CP_q(\CT)$, $\ND_q(\CT)$ and $\RT_q(\CT)$
stand for the spaces of functions whose restriction to each $K \in \CT$ respectively
belong to $\CP_q(K)$, $\ND_q(K)$ and $\RT_q(K)$. Notice that these spaces to do not
exhibit any built-in compatibility conditions.

\subsection{Discrete solution}

In the remainder of this work, we consider a fixed conforming
finite element space $\BW_h \subset \BH_0(\ccurl,\Omega)$. Specifically,
we assume that (i) $\BCP_1(\CT_h) \cap \BH_0(\ccurl,\Omega) \subset \BW_h$ and (ii)
for each $K \in \CT_h$, there exists a value $q_K \geq 1$ such that
$\bv_h|_K \in \ND_{q_K}(K)$ for all $\bv_h \in \BW_h$ and we will denote by
$\pK$ the smallest integer $q_K$ for which the inclusion holds.
The notation $p \eq \min_{K \in \CT_h} \pK$ will be useful.

We further assume that we are given an element $\BE_h \in \BW_h$ satisfying
\begin{equation}
\label{eq_galerkin_orthogonality}
b(\BE_h,\bv_h) = i\omega(\BJ,\bv_h)_\Omega
\end{equation}
for all $\bv_h \in \BW_h$. Notice that if
$\BW_h$ is sufficiently rich (in particular, if $(\omega h)/(p\cminD{\Omega})$
is small enough), the linear system in~\eqref{eq_galerkin_orthogonality} in fact
uniquely defines $\BE_h$. Precise statements may be found in~\cite{%
chaumontfrelet_ern_2023a,%
ern_guermond_2018a,%
melenk_sauter_2020a,%
monk_1992a,%
zhong_shu_wittum_xu_2009a},
where the dependency on the frequency is also analyzed.

\subsection{Right-hand side}

We will assume for the remainder of this work that $\BJ \in \RT_p(\CT_h) \cap \BH(\ddiv,\Omega)$.
The requirement that $\BJ \in \BH(\ddiv,\Omega)$ is crucial, and is already demanded in earlier
works~\cite{chaumontfrelet_vega_2022a}. On the other hand, we only require that
$\BJ \in \RT_p(\CT_h)$ for the sake of simplicity. The general case may be treated,
as usual, by including a ``data oscillation'' term in the estimator, as detailed
in~\cite{chaumontfrelet_vohralik_2023a}. To emphasize that we restrict our attention
to a piecewise polynomial right-hand sides, we employ the notation $\BJ_h$ instead of
$\BJ$ in the remainder of this work.

\subsection{Cohomology}

We introduce
\begin{equation*}
\LBA_{\ee}(\Omega) \eq \left \{
\ee^{-1}\curl \bphi \; | \; \bphi \in \BH_0(\ccurl,\Omega)
\right \}
\end{equation*}
and denote by $\LBA_{\ee}^\perp(\Omega)$ its orthogonal complement in
the $(\ee\cdot,\cdot)_\Omega$ inner product, i.e., $\bv \in \LBA_{\ee}^\perp(\Omega)$
if and only if $\bv \in \BL^2(\Omega)$ with
\begin{equation*}
(\bv,\curl \bphi)_\Omega = 0 \qquad \forall \bphi \in \BH_0(\ccurl,\Omega).
\end{equation*}
When $\Omega$ is simply-connected, we have $\LBA_{\ee}(\Omega) = \BH(\ddiv^0,\ee,\Omega)$
and $\LBA_{\ee}^\perp(\Omega) = \grad H^1_0(\Omega)$. In the general case, the inclusions
$\LBA_{\ee}(\Omega) \subset \BH(\ddiv^0,\ee,\Omega)$ and
$\grad H^1_0(\Omega) \subset \LBA_{\ee}^\perp(\Omega)$
may be strict~\cite{fernandes_gilardi_1997a}.

\subsection{Approximation factor}

The real number
\begin{equation}
\label{eq_approximation_factor}
\gba
\eq
\sup_{\substack{\btheta \in \LBA_{\ee}(\Omega) \\ \|\btheta\|_{\ee,\Omega} = 1}}
\inf_{\bv_h \in \BW_h} \enorm{\bxi_{\btheta}-\bv_h}_{\omega,\Omega},
\end{equation}
where $\bxi_{\btheta}$ is the unique element of $\BH_0(\ccurl,\Omega)$ such that
\begin{equation*}
b(\bxi_{\btheta},\bv) = \omega (\ee\btheta,\bv)_\Omega \qquad \forall \bv \in \BH_0(\ccurl,\Omega),
\end{equation*}
is called the approximation factor. It is instrumental in the a priori and
a posteriori analysis of time-harmonic wave propagation problems~\cite{%
chaumontfrelet_ern_2023a,%
chaumontfrelet_ern_vohralik_2021a,%
chaumontfrelet_vega_2022a,%
dorfler_sauter_2013a,%
sauter_zech_2015a,%
melenk_sauter_2010a,%
melenk_sauter_2020a}.
In fact, the supremum and infimum in~\eqref{eq_approximation_factor}
can be replaced by a maximum and a minimum, and $\gba$ is the best constant
such that
\begin{equation}
\label{eq_estimate_best_approximation}
\min_{\bv_h \in \BW_h}
\enorm{\bxi_{\btheta}-\bv_h}_{\omega,\Omega}
\leq
\gba \|\btheta\|_{\ee,\Omega}
\end{equation}
for all $\btheta \in \LBA_{\ee}(\Omega)$. Taking $\bv_h \eq \bzero$ in
\eqref{eq_approximation_factor}, we see that
\begin{equation}
\label{eq_gba_gst}
\gba \leq \gst
\end{equation}
since
\begin{equation*}
\enorm{\bxi_{\btheta}}_{\omega,\Omega}
\leq
\gst \sup_{\substack{\bv \in \BH_0(\ccurl,\Omega) \\ \enorm{\bv}_{\omega,\Omega} = 1}}
b(\bxi_{\btheta},\bv)
=
\gst \sup_{\substack{\bv \in \BH_0(\ccurl,\Omega) \\ \enorm{\bv}_{\omega,\Omega} = 1}}
\omega (\ee\btheta,\bv)
\leq
\gst \|\btheta\|_{\ee,\Omega}
\end{equation*}
for all $\btheta \in \LBA_{\ee}(\Omega)$.

Because $\btheta \in \LBA_{\ee}(\Omega) \subset \BH(\ddiv^0,\ee,\Omega)$,
the associated solution $\bxi_{\btheta}$ typically exhibits some smoothness,
and there is a rate $s > 0$ such that $\gba \sim \{(\omega h)/(\cminD{\Omega} p)\}^s$.
The value of $s$ and the dependency of the hidden constant in $\sim$ on
the frequency are analyzed, e.g., in~\cite{chaumontfrelet_ern_2023a,%
chaumontfrelet_vega_2022a,%
melenk_sauter_2020a}.
Furthermore, because we restrict the right-hand sides in~\eqref{eq_approximation_factor}
to $\LBA_{\ee}(\Omega)$ (as opposed to $\BH_0(\ddiv^0,\ee,\Omega)$), $\gba$ remains
uniformly bounded as $\omega \to 0$.

\subsection{Vertex patches and hat functions}

We denote by $\CV_h$ the set of the vertices of $\CT_h$.
If $\ba \in \CV_h$, $\CTa \subset \CT_h$ is the submesh
consisting of those elements $K \in \CT_h$ such that $\ba \in \overline{K}$,
and $\oma \eq \operatorname{int} \left (\bigcup_{K \in \CTa} \overline{K}\right )$
is the associated domain, which is always topologically trivial.
We also set $\pa \eq \max_{K \in \CTa} \pK$.

We use the notation $\psia$ for the ``hat function'' associated with $\ba$, i.e.,
the only element of $\CP_1(\CT_h) \cap H^1(\Omega)$ such that
$\psia(\bb) = \delta_{\ba,\bb}$ for all $\bb \in \CV_h$.
We will frequently use the facts that $0 \leq \psia \leq 1$ in $\Omega$,
\begin{equation}
\label{eq_partition_unity}
\sum_{\ba \in \CV_h} \psia = 1 \text{ in } \Omega
\end{equation}
and
\begin{equation}
\label{eq_inverse_hat}
\|\grad \psia\|_{\BL^\infty(\oma)} \lesssim_{\CTa} \ha^{-1},
\end{equation}
where $\|\grad \psia\|_{\BL^\infty(\oma)}$ denotes the essential supremum of $|\grad \psia|$
over $\oma$.

For $\ba \in \CV_h$, the extended vertex patch $\tCTa$ collects all the elements
$K \in \CT_h$ that share at least one vertex with at least one element of $\CTa$.
In effect, $\tCTa$ as one additional layer of elements around $\ba$ as compared
to $\CTa$. $\toma$ is the open domain covered by the elements of $\CTa$.

For $K \in \CT_h$, we also employ the notation $\omK$ for the domain covered by 
all the elements sharing at least one vertex with $K$, and $\tomK$ for the one
covered by all elements sharing at least one vertex with an element contained in
$\omK$. The corresponding sets of elements are respectively denoted by $\CTK$ and
$\tCTK$.

\subsection{Local functional spaces and inequalities}

For a vertex $\ba \in \CV_h$, we define local spaces as follow.
First, if $\ba \in \partial \Omega$, we denote by $\Gamma_{\ba} \subset \partial \Omega$
the portion of the boundary covered by faces sharing the vertex $\ba$,
$\Gamma_{\ba}^{\rm c} \eq \partial \oma \setminus \Gamma_{\ba}$
and $L^2_\star(\oma) \eq L^2(\oma)$. If on the other hand $\ba \in \Omega$,
then $\Gamma_{\ba} \eq \emptyset$, $\Gamma_{\ba}^{\rm c} \eq \partial \Omega$
and $L^2_\star(\oma) \eq \{ v \in L^2(\oma) \; | \; (v,1)_{\oma} = 0 \}$.

For all vertices $\ba \in \CV_h$, we then introduce the spaces
$H^1_\dagger(\oma) \eq \{ v \in H^1(\oma) \; | \; v = 0 \text{ on } \Gamma_{\ba} \}$,
$\BH_0(\ddiv,\oma) \eq \{ \bv \in \BH(\ddiv,\oma) \; | \; \bv \cdot \bn = 0
\text{ on } \Gamma_{\ba}^{\rm c} \}$,
$\BH_\dagger(\ccurl,\oma) \eq \{ \bv \in \BH(\ccurl,\oma) \; | \; \bv \times \bn = \bzero
\text{ on } \Gamma_{\ba} \}$,
$\BH_\star(\ddiv^0,\oma) \eq \{ \bv \in \BH(\ddiv^0,\oma) \; | \; \bv \cdot \bn = 0
\text{ on } \Gamma_{\ba}^{\rm c} \}$,
and
$\BH_0(\ccurl,\oma) \eq \{ \bv \in \BH(\ccurl,\oma) \; | \; \bv \times \bn = \bzero
\text{ on } \Gamma_{\ba}^{\rm c} \}$, where $\bn$ denote the
outward normal unit vector on $\partial \oma$ and the normal and tangential traces
are defined by integration by parts following~\cite{fernandes_gilardi_1997a}.

For all $\ba \in \CV_h$, the following inequalities hold true
\begin{equation}
\label{eq_poincare}
\|v\|_{\oma} \lesssim_{\CTa} \ha \|\grad v\|_{\oma},
\qquad
\|\bw\|_{\oma} \lesssim_{\CTa} \ha \|\curl \bw\|_{\oma}
\end{equation}
for all $v \in H^1_\dagger(\oma) \cap L^2_\star(\oma)$ and
$\bw \in \BH_\dagger(\ccurl,\oma) \cap \BH_\star(\ddiv^0,\oma)$,
see~\cite{veeser_verfurth_2012a} and~\cite[Theorem A.1]{chaumontfrelet_ern_vohralik_2022a}.


Notice that in contrast to $\Omega$, the cohomology of the vertex patches is always trivial,
i.e., the image of the curl operator is the space of divergence-free functions. Specifically,
we have
\begin{equation}
\label{eq_cohomology_patches}
\curl \BH_0(\ccurl,\omega) = \BH_0(\ddiv,\oma)
\end{equation}
for all $\ba \in \CV_h$. This is due to the fact that each $\oma$ is homotopy equivalent
to a ball, and that both the portion on the boundary on which the essential boundary conditions
are imposed and its complement are simply-connected.

\section{Main results}
\label{section_main_results}

Here, we state the main theoretical contribution of this work.

\subsection{Prager-Synge type upper-bound}
\label{section_prager_synge}

In this section, we consider two fields $\BD_h \in \BH(\ddiv,\Omega)$ and
$\BH_h \in \BH(\ccurl,\Omega)$ such that
\begin{equation}
\label{eq_equilibrated_fields}
i\omega \curl \BH_h = i\omega\BJ_h + \omega^2 \BD_h.
\end{equation}
These fields are meant to resemble the electric displacement
$\BD_\star \eq \ee\BE$ and magnetic field $\BH_\star \eq (i\omega\mm)^{-1}\curl \BE_h$,
and the subscript $h$ indicates that they will eventually be constructed
using discrete finite element spaces.

For all $K \in \CT_h$, we define the local estimators
\begin{equation}
\label{eq_local_estimators}
\eta_{\ddiv,K}
\eq
\omega \|\BE_h-\ee^{-1} \BD_h\|_{\ee,K}
\qquad
\eta_{\ccurl,K}
\eq
\|\curl \BE_h-i\omega\cc^{-1}\BH_h\|_{\cc,K}
\end{equation}
and $\eta_K^2 \eq \eta_{\ddiv,K}^2 + \eta_{\ccurl,K}^2$.
We also consider the global versions
\begin{subequations}
\label{eq_global_estimators}
\begin{align}
\eta_{\ddiv}^2
&\eq
\sum_{K \in \CT_h} \eta_{\ddiv,K}^2
=
\omega^2 \|\BE_h-\ee^{-1} \BD_h\|_{\ee,\Omega}^2
\intertext{and}
\eta_{\ccurl}^2
&\eq
\sum_{K \in \CT_h} \eta_{\ccurl,K}^2
=
\|\curl \BE_h-i\omega\cc^{-1}\BH_h\|_{\cc,\Omega}^2
\end{align}
\end{subequations}
as well as $\eta^2 \eq \eta_{\ddiv}^2 + \eta_{\ccurl}^2$.
Our first key result consists of two reliability estimates of Prager-Synge type.
\begin{subequations}
\begin{theorem}[Error estimates in energy norm]
\label{theorem_reliability}
The error estimates
\begin{equation}
\label{eq_estimate_error_energy_coarse}
\enorm{\BE-\BE_h}_{\omega,\Omega}
\leq
\gst \eta
\end{equation}
and
\begin{equation}
\label{eq_estimate_error_energy}
\enorm{\BE-\BE_h}_{\omega,\Omega}
\leq
(1 + 2\gba + 3\gba^2)^{1/2} \eta.
\end{equation}
hold true.
\end{theorem}
\end{subequations}

Estimate~\eqref{eq_estimate_error_energy} is asymptotically constant-free,
since $\gba \to 0$ as $(\omega h)/(p \cminD{\Omega}) \to 0$. Besides,
we obtain a guaranteed bound whenever an estimate for $\gst$ or $\gba$
is available.

\subsection{Local reconstruction of auxiliary fields}
\label{section_local_reconstruction}

We now propose a fully local construction of equilibrated fields
satisfying~\eqref{eq_equilibrated_fields}. These fields are defined
as the sum of local contributions solution to finite element problems
set on vertex patches. The fact that all the local problems appearing
in the definitions below are well-posed is established later in
Section~\ref{section_efficiency}.

\subsubsection{Electric displacement}

We first introduce a local reconstruction designed to mimic the ``true'' electric displacement
$\BD_\star \eq \ee \BE$. Specifically, the key property we want to reproduce is
that $-\omega^2 \div \BD_\star = i\omega \div \BJ_h$.
Following~\cite{destuynder_metivet_1999a}, we remark that
$-\omega^2\div(\psia \BD^\star) = i\omega\psia \div \BJ_h-\omega^2\grad \psia \cdot (\ee\BE)$,
which lead to the definition
\begin{equation}
\label{eq_definition_BD_h_ba}
\BD_h^{\ba}
\eq
\arg \min_{\substack{
\bv_h \in \RT_{\pa+2}(\CTa) \cap \BH_0(\ddiv,\oma)
\\
-\omega^2 \div \bv_h = i\omega \psia \div \BJ_h - \omega^2 \grad \psia \cdot (\ee\BE_h)
}}
\|\psia \BE_h-\ee^{-1}\bv_h\|_{\ee,\oma}.
\end{equation}
We also set
\begin{equation}
\label{eq_definition_BD_h}
\BD_h \eq \sum_{\ba \in \CV_h} \BD_h^{\ba}.
\end{equation}

\subsubsection{Total current variation}

Next, we are interested in the total current variation
\begin{equation*}
\BG_\star \eq \curl(\cc\curl\BE)
=
i\omega \BJ_h + \omega^2 \ee \BE
=
i\omega \BJ_h + \omega^2 \BD_\star
\end{equation*}
and its discrete counter part
\begin{equation*}
\BG_h \eq i\omega \psia \BJ_h + \omega^2 \BD_h.
\end{equation*}
Crucially, $\div \BG_\star = \div \BG_h = 0$, and we would like to obtain a divergence-free
decomposition of $\BG_\star$ and $\BG_h$ into patch-wise contributions $\BG_\star^{\ba}$
and $\BG_h^{\ba}$. We start by remarking that
\begin{equation*}
\BG_\star^{\ba}
\eq
\curl (\psia\cc\curl\BE)
=
i\omega \psia \BJ_h
+
\omega^2 \psia \BD_\star
+
\btheta_\star^{\ba}
\end{equation*}
with
\begin{equation*}
\btheta_\star^{\ba} \eq \grad \psia \times (\cc\curl\BE).
\end{equation*}
It is key to observe that at the continuous level,
\begin{equation*}
\div \btheta_\star^{\ba}
=
-\grad \psia \cdot \curl (\cc \curl \BE),
\end{equation*}
so that $\btheta_\star^{\ba} \in \BH_0(\ddiv,\oma)$ and $\div \BG_\star^{\ba} = 0$.
Unfortunately~\cite{chaumontfrelet_vohralik_2023a}, the ``naive'' discrete counterpart
to $\btheta_\star^{\ba}$, namely $\grad \psia \times (\cc\curl \BE_h)$, does not belong
to $\BH(\ddiv,\oma)$, and in particular, does not lead to a divergence-free decomposition
of $\BG_h$.

Following~\cite{chaumontfrelet_vohralik_2023a}, we thus construct a discrete counterpart
to $\btheta_\star^{\ba}$ in two steps. In the first step, we remark that
\begin{equation*}
\div (\btheta_\star^{\ba})
=
-\grad \psia \cdot \curl(\cc\curl \BE)
=
-\grad \psia \cdot (i\omega \BJ_h + \omega^2 \ee \BE).
\end{equation*}
This motivates the definitions
\begin{equation}
\label{eq_definition_tbtheta_h_ba}
\tbtheta_h^{\ba}
\eq
\arg \min_{\substack{
\bv_h \in \RT_{\pa+1}(\CTa) \cap \BH_0(\ddiv,\oma)
\\
\div \bv_h = -\grad \psia \cdot (i\omega \BJ_h+\omega^2\ee\BE_h)
\\
(\bv_h,\br)_{\oma} = (\grad \psia \times (\cc\curl\BE_h),\br)_{\oma} \; \forall \br \in \BCP_0(\CTa)
}}
\| \grad \psia \times (\cc\curl\BE_h)-\bv_h \|_{\mm,\oma},
\end{equation}
and
\begin{equation*}
\tbtheta_h \eq \sum_{\ba \in \CV_h} \tbtheta_h^{\ba}.
\end{equation*}
As compared to $\btheta_\star^{\ba}$, $\tbtheta_h^{\ba}$ has the desired divergence,
but does not sum up to zero. This is remedied in the second step.
For all $\ba \in \CV_h$ and $K \in \CTa$, we introduce
\begin{equation}
\label{eq_definition_hbtheta_h_ba}
\hbtheta_h^{\ba}|_K
\eq
\arg \min_{\substack{
\bv_h \in \RT_{\pa+2}(K)
\\
\div \bv_h = 0 \text{ in } K
\\
\bv_h \cdot \bn_K = \psia \tbtheta_h \cdot \bn_K \text{ on } \partial K
}}
\|\psia\tbtheta_h-\bv_h\|_{\mm,K}.
\end{equation}
We can now define $\btheta_h^{\ba} \eq \tbtheta_h^{\ba} - \hbtheta_h^{\ba}$, and
the localized total current variation is given by
\begin{equation}
\label{eq_definition_BG_h_ba}
\BG_h^{\ba} \eq i\omega \psia \BJ_h + \omega^2 \BD_h^{\ba} + \btheta_h^{\ba}.
\end{equation}

Notice that the element-wise mean-value constraints in~\eqref{eq_definition_tbtheta_h_ba},
which may appear peculiar at first glance, are actually crucial to ensure that the
minimization problem of second step,~\eqref{eq_definition_hbtheta_h_ba}, is well-posed.
This is detailed in~\cite{chaumontfrelet_vohralik_2023a} and analyzed in Section
\ref{section_total_current_variation} below.

\subsubsection{Magnetic field}

Having properly decomposed the total current density, we can now easily introduce
a local magnetic field reconstruction. Specifically, for each vertex $\ba \in \CV_h$,
the magnetic field is constructed as
\begin{equation}
\label{eq_definition_BH_h_ba}
\BH_h^{\ba}
\eq
\arg \min_{\substack{
\bv_h \in \ND_{\pa+2}(\CTa) \cap \BH_0(\ccurl,\oma)
\\
i\omega \curl \bv_h = \BG_h^{\ba}
}}
\| \psia \curl \BE_h - i\omega \cc^{-1} \bv_h \|_{\cc,\oma}.
\end{equation}
and we set
\begin{equation}
\label{eq_definition_BH_h}
\BH_h \eq \sum_{\ba \in \CV_h} \BH_h^{\ba}.
\end{equation}

\subsubsection{Properties of the local reconstructions}

We now summarize the key properties of our local reconstructions,
namely, that they are well-posed and lead to suitable equilibrated fields.

\begin{theorem}[Equilibrated local reconstructions]
The minimization problems~\eqref{eq_definition_BD_h_ba},
\eqref{eq_definition_tbtheta_h_ba},~\eqref{eq_definition_hbtheta_h_ba}
and~\eqref{eq_definition_BH_h_ba} admit unique minimizers.
The fields $\BD_h$ and $\BH_h$ defined by~\eqref{eq_definition_BD_h} and 
\eqref{eq_definition_BH_h} satisfy the assumptions of Section~\ref{section_prager_synge}
given in~\eqref{eq_equilibrated_fields}.
\end{theorem}

\subsection{Efficiency}

In this section, we specifically consider the estimator $\eta_{\ddiv,K}$ and $\eta_{\ccurl,K}$
defined by~\eqref{eq_local_estimators} with the fields $\BD_h$ and $\BH_h$ obtained
from the construction presented in Section~\ref{section_local_reconstruction}.
Our key result is that the resulting estimator is locally efficient, with a
constant independent of the polynomial degree distribution $p_{\CT_h}$.

\begin{theorem}[Efficiency]
\label{theorem_efficiency}
The estimate
\begin{equation}
\label{eq_efficiency_K}
\eta_K
\lesssim_{\tCTK}
\left (
1 + \frac{\omega h_K}{\cminD{\tomK}}
\right )
\contrasteps{\tomK} \contrastmu{\tomK}^{5/2}
\enorm{\BE-\BE_h}_{\omega,\tomK}
\end{equation}
holds true for all $K \in \CT_h$,
\end{theorem}

This result is established in Section~\ref{section_efficiency}, together
with ``vertex patch versions'' of estimates~\eqref{eq_efficiency_K} where
the (unsummed) local contributions $\BD_h^{\ba}$ and $\BH_h^{\ba}$ appear.

\subsection{Remarks}

We close this section with some remarks on the main results.

(i) For each vertex patch $\ba$, the local contributions to our reconstructions
use uniformly the maximum polynomial degree $\pa$ on the patch $\CTa$. It is possible
to employ slightly cheaper local problems where $p_K$ is used in each element instead
of $\pa$. In this case, we can still show the constant-free reliability of the estimator
and its efficiency, but we are not able to show the $p$-robustness of the lower bound.

(ii) We end up requiring an increase of polynomial degree by $2$ to construct
$\BH_h$, as opposed to the traditional increase by one degree for equilibrated
estimators. This could be remedied by using Brezzi--Douglas--Marini
elements~\cite{brezzi_douglas_marini_1985a} instead of Raviart--Thomas
elements~\cite{raviart_thomas_1977a}, and the second family of N\'ed\'elec
elements~\cite{nedelec_1986a} instead of the first one~\cite{nedelec_1980a}.

(iii) For $\eta_{\ddiv,K}$, the sharper efficiency bound
\begin{equation*}
\eta_{\ddiv,K}
\lesssim_{\CTK}
\contrasteps{\omK}
\omega \|\BE-\BE_h\|_{\ee,\omK}
\end{equation*}
is available for all $K \in \CT_h$, as the proof of Theorem~\ref{theorem_efficiency} shows.

(iv) Our magnetic field reconstruction closely follows the approach
introduced in~\cite{chaumontfrelet_vohralik_2023a} for the case $\omega = 0$.
As opposed to~\cite{chaumontfrelet_vohralik_2023a}, we do require that
$\BCP_1(\CT_h) \cap \BH_0(\ccurl,\Omega) \subset \BW_h$. This condition is 
satisfied for the second-order N\'ed\'element elements of the first-family
and the lowest-order N\'ed\'elec elements of the second family, but not
for the lowest-order N\'ed\'elec element sof the first family. This additional
requirement is due to the presence of the zero-order term in the model problem.

(v) The dependence on the contrast in the magentic permeability, $\contrastmu{\tomK}$
seems suboptimal. In fact, defining $\btheta$ using unweighted norms leads to a linear
scaling in $\contrastmu{\tomK}$ instead of order $5/2$. This is reasdily check by going
through the proof. Nevertheless, we have chosen to defined $\btheta$ with the weights,
since it is the physically relevant norm.

(vi) To be physically consistent, we have added factors $i\omega$ in front of
$\BJ_h$ in the right-hand of \eqref{eq_maxwell_strong} and in front of $\BH_h$
for its definition. Mathematically, however, this factor can be removed. This
may be advantageous since then, if $\BJ_h$ is real-valued, the solution $\BE_h$
and the reconstructions $\BD_h$ and $\BE_h$ may be computed using only real
floating point numbers.

\section{Reliability}
\label{section_reliability}

We start by establishing the reliability properties of our estimator stated in
Theorem~\ref{theorem_reliability}. In this section, the estimators $\eta_{\ddiv}$
and $\eta_{\ccurl}$ of~\eqref{eq_global_estimators} can be obtained from any
reconstructions $\BD_h$ and $\BH_h$ of the electric displacement and the magnetic
field satisfying~\eqref{eq_equilibrated_fields}. In particular, these fields do
not need to be provided by the local construction presented in Section
\ref{section_local_reconstruction}.

Following~\cite{chaumontfrelet_ern_vohralik_2021a,chaumontfrelet_vega_2022a} the
first step is to obtain an upper-bound for the ``residual term'', that is, given
$\bv \in \BH_0(\ccurl,\Omega)$ the quantity $|b(\BE-\BE_h,\bv)|$.

\begin{lemma}[Residual term]
For all $\bv \in \BH_0(\ccurl,\Omega)$, we have
\begin{subequations}
\label{eq_estimates_residual}
\begin{equation}
\label{eq_estimate_residual}
|b(\BE-\BE_h,\bv)|
\leq
\omega \eta_{\ddiv}\|\bv\|_{\ee,\Omega} + \eta_{\ccurl} \|\curl \bv\|_{\cc,\Omega},
\end{equation}
and
\begin{equation}
\label{eq_estimate_residual_energy}
|b(\BE-\BE_h,\bv)|
\leq
\eta \enorm{\bv}_{\omega,\Omega}.
\end{equation}
\end{subequations}
\end{lemma}

\begin{proof}
Let $\bv \in \BH_0(\ccurl,\Omega)$. We have
\begin{align*}
b(\BE-\BE_h,\bv)
&=
(i\omega\BJ_h+\omega^2 \ee\BE_h,\bv) - (\cc\curl \BE_h,\curl \bv)
\\
&=
\omega^2(\ee\BE_h-\BD_h,\bv) + (i\omega\BJ+\omega^2 \BD_h,\bv) - (\cc\curl \BE_h,\curl \bv)
\\
&=
\omega^2(\ee\BE_h-\BD_h,\bv) + i\omega(\curl \BH_h,\bv) - (\cc\curl \BE_h,\curl \bv)
\\
&=
\omega^2(\ee\BE_h-\BD_h,\bv) + (i\omega\BH_h-\cc\curl\BE_h,\curl \bv),
\end{align*}
and~\eqref{eq_estimate_residual} follows from Cauchy-Schwartz inequalities.
Then,~\eqref{eq_estimate_residual_energy} follows from~\eqref{eq_estimate_residual}
since
\begin{equation*}
\omega \eta_{\ddiv} \|\bv\|_{\ee,\Omega} + \eta_{\ccurl}\|\curl \bv\|_{\cc,\Omega}
\leq
\left (
\eta_{\ddiv}^2 + \eta_{\ccurl}^2
\right )^{1/2}
\left (
\omega^2 \|\bv\|_{\ee,\Omega}^2 + \|\curl \bv\|_{\cc,\Omega}^2
\right )^{1/2}.
\end{equation*}
\end{proof}

The next step of the reliability proof consists in employing
the estimates in~\eqref{eq_estimates_residual}, as done in~\cite{%
chaumontfrelet_ern_vohralik_2021a,chaumontfrelet_vega_2022a}.
Specifically, we follow the approach given in~\cite{chaumontfrelet_vega_2022a}
for residual-based estimator, but we take special care to obtain clean constants
as in~\cite{chaumontfrelet_ern_vohralik_2021a} for the Helmholtz equation.

Introducing $\bg$ as the only element of $\LBA_{\ee}^\perp(\Omega)$ such that
\begin{equation*}
(\ee\bg,\bv) = (\ee(\BE-\BE_h),\bv) \quad \forall \bv \in \LBA_{\ee}^\perp(\Omega)
\end{equation*}
we arrive the Helmholtz-Hodge decomposition of the error
\begin{equation*}
\BE-\BE_h = \bg + \btheta
\end{equation*}
with $\btheta \in \BH_0(\ccurl,\Omega) \cap \LBA_{\ee}(\Omega)$. We have
\begin{equation}
\label{eq_helmholtz_hodge_pythagor}
\|\BE-\BE_h\|_{\ee,\Omega}^2 = \|\bg\|_{\ee,\Omega}^2 + \|\btheta\|_{\ee,\Omega}^2.
\end{equation}
The two terms in the right-hand side of~\eqref{eq_helmholtz_hodge_pythagor} are then estimated
separately. The letter $\bg$ is chosen to remind us that in the case of a simply-connected
domain, we infact have $\bg = \grad q$ for some $q \in H^1_0(\Omega)$.

\begin{lemma}[Error estimates in $\BL^2(\Omega)$ norm]
We have
\begin{subequations}
\label{eq_estimates_L2}
\begin{equation}
\label{eq_estimate_theta_L2}
\omega\|\btheta\|_{\ee,\Omega}
\leq
\gba \eta
\end{equation}
and
\begin{equation}
\label{eq_estimate_gradq_L2}
\omega\|\bg\|_{\ee,\Omega}
\leq
\eta_{\ddiv}.
\end{equation}
\end{subequations}
\end{lemma}

\begin{proof}
We introduce $\bxi$ as the only element of $\BH_0(\ccurl,\Omega)$
such that
\begin{equation*}
b(\bw,\bxi) = \omega(\ee\bw,\btheta) \qquad \forall \bw \in \BH_0(\ccurl,\Omega).
\end{equation*}
We then have
\begin{equation*}
\omega \|\btheta\|_{\ee,\Omega}^2
=
(\ee(\BE-\BE_h),\btheta)
=
b(\BE-\BE_h,\bxi)
=
b(\BE-\BE_h,\bxi-\bv_h)
\end{equation*}
for all $\bv_h \in \BW_h$. Recalling that $\btheta \in \LBA_{\ee}(\Omega)$,
\eqref{eq_estimate_residual_energy} and~\eqref{eq_estimate_best_approximation}
show that
\begin{equation*}
\omega \|\btheta\|_{\ee,\Omega}^2
\leq
\gba \eta \|\btheta\|_{\ee,\Omega},
\end{equation*}
and~\eqref{eq_estimate_theta_L2} follows.

On the other hand, since $\curl \bg = \bzero$, we have
\begin{equation*}
\omega^2\|\bg\|_{\ee,\Omega}^2
=
\omega^2(\ee(\BE-\BE_h),\bg)
=
-b(\BE-\BE_h,\bg)
\leq
\omega\eta_{\ddiv} \|\bg\|_{\ee,\Omega},
\end{equation*}
and~\eqref{eq_estimate_gradq_L2} follows.
\end{proof}

We are now fully equipped to establish Theorem~\ref{theorem_reliability}.
This is done by combining~\eqref{eq_estimates_residual},~\eqref{eq_estimates_L2}
and the G\aa rding inequality satisfied by the bilinear form $b(\cdot,\cdot)$.

\begin{proof}[Proof of Theorem~\ref{theorem_reliability}]
Estimate~\eqref{eq_estimate_error_energy_coarse} simply follows
from inf-sup condition~\eqref{eq_inf_sup} combined with~\eqref{eq_estimate_residual_energy}.
For~\eqref{eq_estimate_error_energy}, we have
\begin{align*}
\|\curl \btheta\|_{\cc,\Omega}^2
&=
b(\btheta,\btheta) + \omega^2\|\btheta\|_{\ee,\Omega}^2
\\
&=
b(\BE-\BE_h,\btheta) + \omega^2\|\btheta\|_{\ee,\Omega}^2
\\
&\leq
\omega\eta_{\ddiv} \|\btheta\|_{\ee,\Omega}
+
\eta_{\ccurl} \|\curl \btheta\|_{\cc,\Omega}
+
\omega^2\|\btheta\|_{\ee,\Omega}^2
\\
&\leq
\gba \eta_{\ddiv} \eta
+
\gba^2 \eta^2
+
\eta_{\ccurl} \|\curl \btheta\|_{\cc,\Omega},
\end{align*}
and it follows from Young's inequality that
\begin{equation}
\label{tmp_estimate_curl_theta}
\|\curl \btheta\|_{\cc,\Omega}^2
\leq
\eta_{\ccurl}^2 + 2\gba \eta_{\ddiv}\eta + 2\gba^2 \eta^2.
\end{equation}
Then,~\eqref{eq_estimate_error_energy} is a direct
consequence of~\eqref{eq_estimate_theta_L2},~\eqref{eq_estimate_gradq_L2}
and~\eqref{tmp_estimate_curl_theta} since
\begin{equation*}
\enorm{\BE-\BE_h}_{\omega,\Omega}^2
=
\omega^2 \|\btheta\|_{\ee,\Omega}^2
+
\omega^2\|\bg\|_{\ee,\Omega}^2
+
\|\curl \btheta\|_{\cc,\Omega}^2
\end{equation*}
and $\eta^2 + 2\gba\eta_{\ddiv}\eta + 3\gba^2 \eta^2 \leq (1 + 2\gba + 3\gba^2)^2 \eta^2$.
\end{proof}

\section{Efficiency}
\label{section_efficiency}

In this section, we show that the localized reconstruction of the fields $\BD_h$ and $\BH_h$
introduced in Section~\ref{section_local_reconstruction} are well-defined. We then establish
Theorem~\ref{theorem_efficiency} which states that the corresponding estimators
$\{\eta_K\}_{K \in \CT_h}$ are locally efficient.

\subsection{Electric displacement reconstruction}

We first address the electric displacement $\BD_h$. Recall that it is defined
as the sum of patch-wise contributions $\BD_h^{\ba}$ that solve the local
minimization problems in~\eqref{eq_definition_BD_h_ba}. Following~\cite{ern_vohralik_2015a},
we start by analyzing a ``continuous version'' of~\eqref{eq_definition_BD_h_ba},
where the Raviart--Thomas space is replaced by the full $\BH(\ddiv,\oma)$ space.

\begin{lemma}[Idealized electric displacement reconstruction]
\label{lemma_BD_ba}
For all $\ba \in \CV_h$, there exists a unique minimizer
\begin{equation}
\label{eq_definition_BD_ba}
\BD^{\ba}
\eq
\arg \min_{\substack{
\bv \in \BH_0(\ddiv,\oma)
\\
-\omega^2 \div \bv = i\omega\psia\div \BJ_h-\omega^2\grad \psia \cdot(\ee\BE_h)
}}
\|\psia \BE_h-\ee^{-1}\bv\|_{\ee,\oma}.
\end{equation}
In addition, we have
\begin{equation}
\label{eq_estimate_BD_ba}
\|\psia\BE_h-\ee^{-1}\BD^{\ba}\|_{\ee,\oma}
\lesssim
\sqrt{\contrasteps{\oma}}
\|\BE-\BE_h\|_{\ee,\oma}.
\end{equation}
\end{lemma}

\begin{proof}
We first check that the minimization set of~\eqref{eq_definition_BD_ba} is non-empty.
If $\ba \in \partial \Omega$, the constraint on the normal trace is not imposed on
the whole boundary of the patch, so that it is always non-empty. For $\ba \in \Omega$,
the usual compatibility condition
\begin{align*}
(i\omega\psia\div \BJ_h-\omega^2 \grad \psia \cdot(\ee\BE_h),1)_{\oma}
&=
i\omega(\div \BJ_h,\psia)_{\oma} - \omega^2 (\ee\BE_h,\grad \psia)_{\oma}
\\
&=
-i\omega(\BJ_h,\grad \psia)_{\oma} + b(\BE_h,\grad \psia)
\\
&=
0
\end{align*}
holds true, since $\grad \psia \in \ND_0(\CT_h) \cap \BH_0(\ccurl,\Omega) \subset \BW_h$.
Then, the existence and uniqueness of the minimizer $\BD^{\ba}$ follows from
standard convexity argument.

The Euler-Lagrange equations associated with~\eqref{eq_definition_BD_ba} show
the existence (and uniqueness) of $q \in L^2_\star(\oma)$ such that
\begin{equation*}
\left \{
\begin{array}{rcl}
(\psia\BE_h-\ee^{-1}\BD^{\ba},\bv)_{\oma} + (q,\div \bv)_{\oma}
&=&
0
\\
-\omega^2 (\div \BD^{\ba},r)_{\oma}
&=&
(i\omega\psia \div \BJ_h-\omega^2 \grad \psia \cdot(\ee\BE_h),r)_{\oma}
\end{array}
\right .
\end{equation*}
for all $\bv \in \BH_0(\ddiv,\oma)$ and $r \in L^2_\star(\oma)$,
and we readily see that $q \in H^1_\dagger(\oma) \cap L^2_\star(\oma)$, with
\begin{equation}
\label{tmp_id_gradq}
-\omega^2 \ee\grad q = -\omega^2 \psia \ee\BE_h + \omega^2 \BD^{\ba}.
\end{equation}
As a result, we have
\begin{equation*}
\|\psia \BE_h-\ee^{-1} \BD^{\ba}\|_{\ee,\oma} = \|\grad q\|_{\ee,\oma}.
\end{equation*}
On the other hand, taking the $\BL^2(\Omega)$ inner product of both sides of
\eqref{tmp_id_gradq} with $\grad q$, we have
\begin{equation*}
-\omega^2 \|\grad q\|_{\ee,\oma}^2
=
(i\omega\psia \div \BJ_h-\omega^2 \grad \psia \cdot(\ee\BE_h),q)_{\oma}
+
\omega^2(\psia\ee\BE_h,\grad q)_{\oma},
\end{equation*}
since $\omega^2 \div \BD^{\ba} = i\omega \psia \div \BJ_h - \omega^2 \grad \psia \cdot (\ee\BE_h)$.
We then remark that
\begin{align*}
\omega^2 (\psia \ee\BE,\grad q)_{\oma}
=
-\omega^2 (\div (\psia \ee\BE),q)_{\oma}
&=
-\omega^2(\psia \div (\ee\BE),q)_{\oma}
-\omega^2 (\grad \psia \cdot (\ee\BE),q)_{\oma}
\\
&=
i\omega(\psia \div \BJ_h,q)_{\oma}
-\omega^2 (\grad \psia \cdot (\ee\BE),q)_{\oma},
\end{align*}
and therefore
\begin{equation*}
-\omega^2 \|\grad q\|_{\ee,\oma}^2
=
\omega^2 (\grad \psia \cdot (\ee\BE-\BE_h),q)_{\oma}
-
\omega^2 (\psia\ee(\BE-\BE_h),\grad q)_{\oma}.
\end{equation*}
It follows that
\begin{equation*}
\|\grad q\|_{\ee,\oma}^2
\leq
\|\BE-\BE_h\|_{\ee,\oma}\|q\grad \psia-\psia\grad q\|_{\ee,\oma}.
\end{equation*}
The conclusion then follow since on the one hand,
\begin{align*}
\|q\grad \psia-\psia\grad q\|_{\ee,\oma}
&\leq
\|q\grad \psia\|_{\ee,\oma}+\|\psia\grad q\|_{\ee,\oma}
\\
&\leq
\sqrt{\epsmaxa}\|q\grad \psia\|_{\oma} + \|\grad q\|_{\ee,\oma}
\\
&\leq
\sqrt{\epsmaxa} \|\grad \psia\|_{\BL^\infty(\oma)}\|q\|_{\oma} + \|\grad q\|_{\ee,\oma},
\end{align*}
and one the other hand
\begin{equation*}
\|q\|_{\oma}
\lesssim_{\CTa}
\ha \|\grad q\|_{\oma}
\leq
\ha \frac{1}{\sqrt{\epsmina}}\|\grad q\|_{\ee,\oma}.
\end{equation*}
where we employed~\eqref{eq_poincare} and~\eqref{eq_inverse_hat}.
\end{proof}

\begin{theorem}[Discrete electric displacement reconstruction]
Problem~\eqref{eq_definition_BD_h_ba} admits a unique minimizer.
In addition, we have
\begin{equation}
\label{eq_estimate_BD_h_ba}
\|\psia \BE_h - \ee^{-1} \BD_h^{\ba}\|_{\ee,\oma}
\lesssim_{\CTa}
\contrasteps{\oma} \|\BE-\BE_h\|_{\ee,\oma}.
\end{equation}
\end{theorem}

\begin{proof}
Let us denote by $g^{\ba} \eq i\omega\psia\div \BJ_h-\omega^2\grad \psia\cdot(\ee\BE_h)$
the divergence constraint in~\eqref{eq_definition_BD_h_ba}.
We already showed in Lemma~\ref{lemma_BD_ba} that the divergence constraint
satisfies the compatibility condition
$(g^{\ba},1)_{\oma} = 0$ when $\ba \in \Omega$. Since in addition
$g^{\ba} \in \CP_{p+1}(\CTa)$, the minimization set of
\eqref{eq_definition_BD_h_ba} is non-empty, and the existence and uniqueness
of $\BD_h^{\ba}$ follows from convexity considerations.

The same reasoning also applies to the problems with unweighted norms,
and we can introduce
\begin{equation*}
\bv^\star \eq \arg \min_{\substack{
\bv \in \BH_0(\ddiv,\oma)
\\
-\omega^2\div \bv = i\omega \BJ_h-\omega^2\psia \cdot(\ee\BE_h)
}}
\|\psia \ee\BE_h-\bv\|_{\oma}
\end{equation*}
and
\begin{equation*}
\bv_h^\star \eq \arg \min_{\substack{
\bv_h \in \RT_{p+2}(\oma) \cap \BH_0(\ddiv,\oma)
\\
-\omega^2\div \bv_h = i\omega \BJ_h-\omega^2\psia \cdot(\ee\BE_h)
}}
\|\psia \ee \BE_h-\bv_h\|_{\oma}.
\end{equation*}
Observing in addition that $\psia\BE_h \in \BCP_{p+2}(\CTa) \subset \RT_{p+2}(\CTa)$,
\cite[Proposition 3.1]{chaumontfrelet_vohralik_2022b} (see also~\cite{ern_vohralik_2021a})
shows that
\begin{equation*}
\|\psia \ee \BE_h-\bv^\star_h\|_{\oma}
\lesssim_{\CTa}
\|\psia \ee \BE_h- \bv^\star\|_{\oma}.
\end{equation*}
We can now conclude since on the one hand
\begin{multline*}
\|\psia \BE_h-\ee^{-1}\BD_h^{\ba}\|_{\ee,\oma}^2
\leq
\|\psia \BE_h-\ee^{-1}\bv_h^\star\|_{\ee,\oma}^2
\\
=
\|\psia \ee \BE_h-\bv_h^\star\|_{\ee^{-1},\oma}^2
\leq
\frac{1}{\epsminD{\oma}}\|\psia \ee \BE_h-\bv_h^\star\|_{\oma}^2
\end{multline*}
and on the other hand
\begin{equation*}
\|\psia \ee \BE_h-\bv^\star\|_{\oma}
\leq
\|\psia \ee \BE_h-\BD^{\ba}\|_{\oma}
\leq
\sqrt{\epsmaxa}
\|\psia \BE_h-\ee^{-1}\BD^{\ba}\|_{\ee,\oma},
\end{equation*}
where $\BD^{\ba}$ is the minimizer in the left-hand side of~\eqref{eq_definition_BD_ba}.
We now conclude with~\eqref{eq_estimate_BD_ba}.
\end{proof}

\subsection{Total current variation}
\label{section_total_current_variation}

We next analyze the total current variation $\BG_h^{\ba}$ from~\eqref{eq_definition_BG_h_ba}.
The key point is the intermediate quantity $\btheta_h^{\ba}$. The proofs of this section
heavily rely on~\cite[Appendix A]{chaumontfrelet_vohralik_2023a}. For $\ba \in \CV_h$ we
thus introduce
\begin{equation}
\label{eq_definition_g_tau}
g^{\ba} \eq -\grad \psia \cdot \left (i\omega\BJ_h + \omega^2 \ee \BE_h\right )
\qquad
\btau_h^{\ba} \eq \grad \psia \times (\cc \curl \BE_h)
\end{equation}
for the two constraints appearing in~\eqref{eq_definition_tbtheta_h_ba}.
We start by introducing a continuous version of $\tbtheta_h^{\ba}$.

\begin{lemma}[Idealized intermediate problem]
There exists a unique minimizer
\begin{equation}
\label{eq_definition_tbtheta_ba}
\tbtheta^{\ba}
\eq
\arg \min_{\substack{\bv \in \BH_0(\ddiv,\oma) \\ \div \bv = g^{\ba}}}
\|\bv - \btau_h^{\ba}\|_{\mm,\oma},
\end{equation}
and the following estimate holds true:
\begin{equation}
\label{eq_estimate_tbtheta_ba}
\|\tbtheta^{\ba}-\btau_h^{\ba}\|_{\mm,\oma}
\lesssim_{\CTa}
\sqrt{\contrastmu{\oma}}
\ha^{-1} \left (
1 + \frac{\omega \ha}{\cmina}
\right )
\enorm{\BE-\BE_h}_{\omega,\oma}.
\end{equation}
\end{lemma}

\begin{proof}
The existence and uniqueness of the minimizer $\tbtheta^{\ba}$ in
\eqref{eq_definition_tbtheta_ba} follows from the usual compatibility
condition. If $\ba \in \partial \Omega$, there is nothing to check. On
the other hand, if $\ba \notin \partial \Omega$, then $\grad \psia \in \BW_h$,
and therefore
\begin{equation*}
(g^{\ba},1)_{\oma}
=
-i\omega(\BJ_h,\grad \psia)_{\oma} + \omega^2 (\ee\BE_h,\grad \psia)_{\oma}
=
-i\omega(\BJ_h,\grad \psia)_\Omega + b(\BE_h,\grad \psia)
=
0.
\end{equation*}

The Euler-Lagrange equations for the right-hand side of
\eqref{eq_definition_tbtheta_ba} ensures the existence
and uniqueness of $q \in L^2_\star(\oma)$ such that
\begin{equation*}
\left \{
\begin{array}{rcl}
(\mm(\tbtheta^{\ba}-\btau_h^{\ba}),\bw)_{\oma} + (q,\div \bw)_{\oma} &=& 0
\\
(\div \tbtheta^{\ba},r)_{\oma} &=& (g^{\ba},r)_{\oma}
\end{array}
\right .
\end{equation*}
for all $\bw \in \BH_0(\ddiv,\oma)$ and $r \in L^2(\oma)$. We easily deduce that
$q \in H^1_\dagger(\oma) \cap L^2_\star(\oma)$ with
$\cc\grad q = \tbtheta^{\ba}-\btau_h^{\ba}$.
Then, we have
\begin{align*}
\|\grad q\|_{\cc,\oma}^2
&=
(\tbtheta^{\ba}-\btau_h^{\ba},\grad q)_{\oma}
=
-(\div \tbtheta^{\ba},q)_{\oma}-(\btau_h^{\ba},\grad q)_{\oma}
=
-(g^{\ba},q)_{\oma}-(\btau_h^{\ba},\grad q)_{\oma}
\\
&=
((i\omega \BJ_h + \omega^2 \ee \BE_h) \cdot \grad \psia,q)_{\oma}
-
((\cc \curl \BE_h) \times \grad \psia,\grad q)_{\oma}
\\
&=
((i\omega \BJ_h + \omega^2 \ee \BE_h),q \grad \psia)_{\oma}
-
((\cc \curl \BE_h),\grad \psia \times \grad q)_{\oma}
\\
&=
i\omega(\BJ_h,q\grad \psia)
+ \omega^2 (\ee \BE_h,q \grad \psia)_{\oma}
-
((\cc \curl \BE_h),\curl(q\grad \psia))_{\oma}
\\
&=
i\omega(\BJ_h,q\grad \psia) - b(\BE-\BE_h,q\grad \psia)_{\oma}
\\
&=
b(\BE-\BE_h,q\grad \psia)_{\oma}
\\
&\leq
\enorm{\BE-\BE_h}_{\omega,\oma} \enorm{q\grad \psia}_{\omega,\oma},
\end{align*}
and since $\|\grad q\|_{\cc,\oma} = \|\tbtheta^{\ba}-\btau_h^{\ba}\|_{\mm,\oma}$,%
~\eqref{eq_estimate_tbtheta_ba} follows with
\begin{align*}
\enorm{q\grad \psia}_{\omega,\oma}
&\lesssim_{\CTa}
\omega \|q\grad \psia\|_{\ee,\oma}
+
\|\grad q \times \grad \psia\|_{\cc,\oma}
\\
&\lesssim_{\CTa}
\|\grad \psia\|_{\BL^\infty(\oma)}
\left (\omega \sqrt{\epsmaxa} \|q\|_{\oma} + \sqrt{\chimaxa} \|\grad q\|_{\oma} \right )
\\
&\lesssim_{\CTa}
\sqrt{\chimaxa}\|\grad \psia\|_{\BL^\infty(\oma)}
\left (1 + \omega h_{\oma} \frac{\sqrt{\epsmaxa}}{\sqrt{\chimaxa}}\right )
\|\grad q\|_{\oma}
\\
&\lesssim_{\CTa}
\sqrt{\contrastmu{\oma}} \|\grad \psia\|_{\BL^\infty(\oma)}
\left (1 + \omega h_{\oma} \frac{\sqrt{\epsmaxa}}{\sqrt{\chimaxa}}\right )
\|\grad q\|_{\cc,\oma}.
\end{align*}
\end{proof}

\begin{lemma}[Discrete intermediate problem]
Problem~\eqref{eq_definition_tbtheta_h_ba} admits a unique minimizer $\tbtheta_h^{\ba}$.
In addition, we have
\begin{equation}
\label{eq_estimate_tbtheta_h_ba}
\|\grad \psia \times (\cc\curl\BE_h) - \tbtheta_h^{\ba}\|_{\mm,\oma}
\lesssim_{\CTa}
\contrastmu{\oma} 
\ha^{-1}
\left ( 1+ \frac{\omega \ha}{\cmina}\right )
\enorm{\BE-\BE_h}_{\omega,\oma}.
\end{equation}
\end{lemma}

\begin{proof}
Recalling the definition of $g^{\ba}$ and $\btau_h^{\ba}$ at~\eqref{eq_definition_g_tau},
we see that Problem~\eqref{eq_definition_tbtheta_h_ba} takes the form
\begin{equation*}
\min_{\substack{
\bv_h \in \RT_{\pa+1}(\CTa) \cap \BH_0(\ddiv,\oma)
\\
\div \bv_h = g^{\ba}
\\
(\bv_h,\br)_{\oma} = (\btau_h^{\ba},\br)_{\oma} \; \forall \br \in \BCP_0(\CTa)
}}
\|\bv_h - \btau_h^{\ba}\|_{\mm,\oma},
\end{equation*}
treated in~\cite[Appendix A]{chaumontfrelet_vohralik_2023a}.
The analysis in~\cite{chaumontfrelet_vohralik_2023a}
is performed without the weight $\mm$ in the norm, but
the analysis holds true here, up to adding $\sqrt{\contrastmu{\oma}}$
as a multplicative factor.
As a result, we simply need to check that the conditions
in~\cite[Assumption A.1]{chaumontfrelet_vohralik_2023a} are satisfied.
First (i), it is clear that $g^{\ba} \in L^2(\oma)$. Second (ii),
if $\ba \notin \partial \Omega$, then $\grad \psia \in \BW_h$,
and we have
\begin{equation*}
(g^{\ba},1)_{\oma}
=
-\omega^2(\ee\BE_h,\grad \psia) - i\omega(\BJ_h,\grad \psia)
=
b(\BE_h,\grad \psia) - i\omega(\BJ_h,\grad \psia) = 0.
\end{equation*}
Finally (iii), fix $q_h \in \CP_1(\CTa) \cap H^1_\dagger(\oma) \cap L^2_\star(\oma)$,
and let $\bw_h \eq q_h \grad \psia$. Since $\grad \psia \in \ND_0(\CT_h) \cap \BH(\ccurl,\oma)$,
it is clear that $\bw_h \in \BCP_1(\CTa) \cap \BH(\ccurl,\oma)$. In addition, combining the
boundary conditions satisfied by $\grad \psia$ and $q_h$, we see that
$\bw_h \in \BCP_1(\CTa) \cap \BH_0(\ccurl,\oma) \subset \BW_h$. As a result,
we can employ~\eqref{eq_galerkin_orthogonality}, and since
$\curl \bw_h = \grad q_h \times \grad \psia$, we have
\begin{multline*}
(\btau_h^{\ba},\grad q_h)_{\oma}
=
(\cc\curl\BE_h \times \grad \psia,\grad q_h)
=
(\cc\curl\BE_h,\curl \bw_h)
\\
=
i\omega(\BJ_h,\bw_h) + \omega^2 (\ee\BE_h,\bw_h)
=
-(g^{\ba},q_h)_{\oma}.
\end{multline*}

We can thus invoke~\cite[Theorem A.2]{chaumontfrelet_vohralik_2023a}, which ensures
the existence and uniqueness of $\tbtheta_h^{\ba}$ and provides the estimate
\begin{equation*}
\|\tbtheta_h^{\ba} - \btau_h^{\ba}\|_{\mm,\oma}
\lesssim_{\CTa}
\sqrt{\contrastmu{\oma}}
\|\tbtheta^{\ba} - \btau_h^{\ba}\|_{\mm,\oma},
\end{equation*}
where, again, the presence of $\contrastmu{\oma}$ follows from the use of the
weighted norm. Then,~\eqref{eq_estimate_tbtheta_h_ba} follows from~\eqref{eq_estimate_tbtheta_ba}.
\end{proof}

\begin{lemma}[Key properties of $\hbtheta_h^{\ba}$]
Problem~\eqref{eq_definition_hbtheta_h_ba} admits a unique minimizer $\hbtheta_h^{\ba}|_K$.
For each $\ba \in \CV_h$, the resulting function $\hbtheta_h^{\ba}$ satisfies
\begin{equation*}
\hbtheta_h^{\ba} \in \BH_0(\ddiv,\oma) \qquad \div \hbtheta_h^{\ba} = 0,
\end{equation*}
and we have
\begin{equation*}
\tbtheta_h = \sum_{\ba \in \CV_h} \hbtheta_h^{\ba}.
\end{equation*}
In addition, the estimate
\begin{equation}
\label{eq_estimate_hbtheta_h_ba}
\|\hbtheta_h^{\ba}\|_{\mm,K} \lesssim_{\CTa} \sqrt{\contrastmu{K}}\|\tbtheta_h\|_{\mm,K}
\end{equation}
holds true.
\end{lemma}

\begin{proof}
The above results are immediate consequences of~\cite[Theorem B.1]{chaumontfrelet_vohralik_2023a}.
As above, the results in~\cite{chaumontfrelet_vohralik_2023a} are given for an unweighted norm,
which is corrected here by the presence of the multplicative factor $\sqrt{\contrastmu{K}}$
in~\eqref{eq_estimate_hbtheta_h_ba}. We rapidly check the assumptions of theorem, namely that
$\tbtheta_h \in \BH(\ddiv,\oma)$ and that for all elements
$K \in \CT_h$
\begin{align*}
\div \tbtheta_h|_K
&=
-\div \left (\sum_{\ba \in \CV(K)} [
\grad \psia \cdot (i\omega \BJ_h + \omega^2 \ee \BE_h)
]\right )
\\
&=
-\div \left (\left (\sum_{\ba \in \CV(K)}
\grad \psia\right ) \cdot (i\omega \BJ_h + \omega^2 \ee \BE_h)
\right )
\\
&=
-\div \left (\left (\grad 1\right ) \cdot (i\omega \BJ_h + \omega^2 \ee \BE_h) \right )
=
0,
\end{align*}
and
\begin{align*}
(\tbtheta_h,\br)_K
=
\sum_{\ba \in \CV_h} (\tbtheta_h^{\ba},\br)_K
=
0
\end{align*}
for all $\br \in \BCP_0(K)$.
\end{proof}

We are now ready to give an estimate to the mismatch between $\BG_h^{\ba}$
and its continuous counterpart
\begin{equation*}
\BG^{\ba} \eq i\omega \psia \BJ_h + \omega^2\psia\ee\BE + \grad \psia \times (\cc\curl \BE).
\end{equation*}

\begin{theorem}[Total current variation reconstruction]
\label{theorem_total_current}
We have $\BG_h^{\ba} \in \RT_{\pa+2}(\CTa) \cap \BH_0(\ddiv,\oma)$ with
\begin{equation}
\label{eq_BG_h_ba_divergence_free}
\div \BG_h^{\ba} = 0.
\end{equation}
In addition, the following estimate holds true:
\begin{equation}
\label{eq_estimate_BG_h_ba}
\|\BG^{\ba}-\BG_h^{\ba}\|_{\mm,\oma}
\lesssim_{\tCTa}
\ha^{-1} \contrastmu{\toma}^{3/2} \contrasteps{\toma}
\left (
1 + \frac{\omega\ha}{\cminD{\toma}}
\right )
\enorm{\BE-\BE_h}_{\omega,\toma}.
\end{equation}
\end{theorem}

\begin{proof}
The first part of the statement immediately follows from the fact that
$\psia \BJ_h \in \RT_{\pa+1}(\CTa) \cap \BH_0(\ddiv,\oma)$,
$\BD_h^{\ba} \in \RT_{\pa+2}(\CTa) \cap \BH_0(\ddiv,\oma)$
and
$\btheta_h^{\ba} \in \RT_{\pa+2}(\CTa) \cap \BH_0(\ddiv,\oma)$.
Then, recalling the divergence constraints in
\eqref{eq_definition_BD_h_ba} and~\eqref{eq_definition_tbtheta_h_ba},
we have
\begin{multline*}
\div \BG_h^{\ba}
=
i\omega\div (\psia \BJ_h) + \omega^2 \div (\BD_h^{\ba}) + \div \btheta_h^{\ba}
=
\\
i\omega(\grad \psia \cdot \BJ_h + \psia \div \BJ_h)
-(i\omega \psia \div \BJ_h - \omega^2 \grad \psia \cdot (\ee \BE_h))
\\
-\grad \psia \cdot(i\omega \BJ_h + \omega^2 \ee \BE_h)
=
0.
\end{multline*}
and~\eqref{eq_BG_h_ba_divergence_free} follows.

We now establish~\eqref{eq_estimate_BG_h_ba}. We have
\begin{equation*}
\BG^{\ba}-\BG_h^{\ba}
=
\omega^2 (\psia \ee\BE-\BD_h^{\ba})
+
(\grad \psia \times (\cc\curl\BE)-\tbtheta_h^{\ba}) + \hbtheta_h^{\ba}
\end{equation*}
In addition, we have
\begin{equation*}
\omega^2 (\psia\ee\BE-\BD_h^{\ba})
=
\omega^2 \psia (\ee\BE-\ee\BE_h)
+
\omega^2(\psia\ee\BE_h-\BD_h^{\ba}),
\end{equation*}
and using~\eqref{eq_estimate_BD_h_ba}, we have
\begin{align}
\label{tmp_estimate_psia_ee_BE_BD_h_ba}
\omega^2 \|\psia \ee\BE-\BD_h^{\ba}\|_{\mm,\oma}
&\leq
\sqrt{\frac{\epsmaxa}{\chimina}} \omega^2 \left (
\|\BE-\BE_h\|_{\ee,\oma} + \|\psia \BE_h-\ee^{-1} \BD_h^{\ba}\|_{\ee,\oma}
\right )
\\
\nonumber
&\lesssim_{\CTa}
\contrasteps{\oma} \sqrt{\frac{\epsmaxa}{\chimina}} \omega^2 \|\BE-\BE_h\|_{\ee,\oma}
\\
\nonumber
&=
\ha^{-1} 
\contrasteps{\oma}\frac{\omega \ha}{\cmina}
\omega \|\BE-\BE_h\|_{\ee,\oma}.
\end{align}
Recalling~\eqref{eq_estimate_hbtheta_h_ba}, we have
\begin{multline*}
\|\hbtheta_h^{\ba}\|_{\mm,K}
\lesssim_{\CTa}
\sqrt{\contrastmu{K}}
\|\tbtheta_h\|_{\mm,K}
=
\sqrt{\contrastmu{K}}
\left \|
\sum_{\bb \in \CV_h}
\left (
\tbtheta_h^{\bb} -
\grad \psi^{\bb} \times (\cc \curl \BE_h)
\right )
\right \|_{\mm,K}
\\
\lesssim_{\tCTa}
\sqrt{\contrastmu{K}}
\sum_{\bb \in \CV_h}
\hb^{-1}
\contrastmu{\omb}
\left (
1 + \frac{\omega h_{\omega^{\bb}}}{\cminb}
\right )\enorm{\BE-\BE_h}_{\omega,\omb},
\end{multline*}
and summing over $K \in \CTa$, since $h_{\omega^{\bb}} \lesssim_{\tCTa} h_{\oma}$
for all vertices $\bb$ in the patch, we arrive at
\begin{equation}
\label{tmp_estimate_hbtheta_h_ba}
\|\hbtheta_h^{\ba}\|_{\mm,\oma}
\lesssim
\contrastmu{\toma}^{3/2}
h_{\oma}^{-1}
\left (
1 + \frac{\omega h_{\oma}}{\cminD{\toma}}
\right )\enorm{\BE-\BE_h}_{\omega,\toma}.
\end{equation}
Then,~\eqref{eq_estimate_BG_h_ba} follows from~\eqref{eq_estimate_tbtheta_h_ba},
\eqref{tmp_estimate_psia_ee_BE_BD_h_ba} and~\eqref{tmp_estimate_hbtheta_h_ba}.
\end{proof}

\subsection{Magnetic fields reconstruction}

We now analyze the magnetic field reconstruction $\BH_h^{\ba}$.
Here too, we start by introducing a continuous version
$\BH^{\ba}$ of $\BH_h^{\ba}$ defined with a minimization
problem over $\BH_0(\ccurl,\oma)$.

\begin{lemma}[Idealized magnetic field reconstruction]
For all $\ba \in \CV_h$, there exists a unique minimizer
\begin{equation}
\label{eq_definition_BH_ba}
\BH^{\ba}
\eq
\arg \min_{\substack{
\bv \in \BH_0(\ccurl,\oma)
\\
i\omega \curl \bv = \BG_h^{\ba}
}}
\|\psia \curl \BE_h-i\omega\cc^{-1} \bv\|_{\cc,\oma}
\end{equation}
In addition, we have
\begin{equation}
\label{eq_estimate_BH_ba}
\|\psia \curl \BE_h - i\omega \cc^{-1} \BH^{\ba}\|_{\cc,\oma}
\lesssim_{\CTa}
\|\curl(\BE-\BE_h)\|_{\cc,\oma}
+
\sqrt{\contrastmu{\oma}}\ha \|\BG^{\ba}-\BG_h^{\ba}\|_{\oma}.
\end{equation}
\end{lemma}

\begin{proof}
Recalling Theorem~\ref{theorem_total_current},
$\BG_h^{\ba} \in \BH_0(\ddiv,\oma)$ with $\div \BG_h^{\ba} = 0$.
Due to~\eqref{eq_cohomology_patches}, this proves that the minimization
set in~\eqref{eq_definition_BH_ba} is non-empty and ensures the existence
and uniqueness of $\BH^{\ba}$.

The Euler-Lagrange conditions associated with the minimization problem
ensures that there exists a (unique) $\bphi \in \BH_\star(\ddiv^0,\oma)$
such that
\begin{equation*}
\left \{
\begin{array}{rcl}
(\psia \curl \BE_h-i\omega \cc^{-1} \BH^{\ba},\bv)_{\oma}
+
(\bphi,\curl \bv)_{\oma}
&=&
0
\\
i\omega (\curl \BH^{\ba},\bw)_{\oma}
&=&
(\BG_h^{\ba},\bw)_{\oma}
\end{array}
\right .
\end{equation*}
for all $\bv \in \BH_0(\ccurl,\oma)$ and $\bw \in \BH_\star(\ddiv^0,\oma)$.
Integrating by parts in the first equation reveals that
$\bphi \in \BH_\dagger(\ccurl,\oma) \cap \BH_\star(\ddiv^0,\oma)$,
with
\begin{equation*}
\curl \bphi = \psia \curl \BE_h - i\omega \cc^{-1} \BH^{\ba}.
\end{equation*}
Similarly, integration by parts in the second equation gives
\begin{equation*}
-i\omega (\BH^{\ba},\curl \bw)_{\oma} = -(\BG_h^{\ba},\bw)_{\oma},
\end{equation*}
and we deduce that
\begin{equation*}
(\cc\curl \bphi,\curl \bw)_{\oma}
=
(\psia \cc \curl \BE_h,\curl \bw)_{\oma} - (\BG_h^{\ba},\bw)_{\oma}.
\end{equation*}
On the other hand, we have
\begin{equation*}
(\psia \cc \curl \BE,\curl \bw) - (\BG^{\ba},\bw)_{\oma} = 0,
\end{equation*}
so that
\begin{equation*}
(\cc\curl\bphi,\curl \bw)_{\oma}
=
(\BG^{\ba}-\BG_h^{\ba},\bw)_{\oma}
-
(\psia\cc\curl(\BE-\BE_h),\curl \bw)_{\oma}.
\end{equation*}
It follows that
\begin{equation*}
\|\curl \bphi\|_{\cc,\oma}^2
\leq
\|\BG^{\ba}-\BG_h^{\ba}\|_{\mm,\oma}\|\bphi\|_{\cc,\oma}
+
\|\curl(\BE-\BE_h)\|_{\cc,\oma}\|\psia\curl \bphi\|_{\cc,\oma},
\end{equation*}
and we conclude since recalling~\eqref{eq_poincare}, we have
$\|\psia \curl \bphi\|_{\cc,\oma} \leq \|\curl \bphi\|_{\cc,\oma}$ and
\begin{equation*}
\frac{1}{\sqrt{\chimaxa}} \|\bphi\|_{\cc,\oma}
\leq
\|\bphi\|_{\oma}
\lesssim_{\CTa}
h_{\oma} \|\curl \bphi\|_{\oma}
\leq
\frac{h_{\oma}}{\sqrt{\chi_{\min,\oma}}} \|\curl \bphi\|_{\cc,\oma}.
\end{equation*}
\end{proof}

\begin{theorem}[Discrete magnetic field reconstruction]
\label{theorem_magnetic_reconstruction}
Problem~\eqref{eq_definition_BH_h_ba} admits a unique minimizer $\BH_h^{\ba}$.
In addition, we have
\begin{equation}
\label{eq_estimate_BH_h_ba}
\|\psia \curl \BE_h-i\omega\cc\BH_h^{\ba}\|_{\cc,\oma}
\lesssim_{\tCTa}
\left (
1 + \frac{\omega \ha}{\cminD{\toma}}
\right )
\contrasteps{\toma} \contrastmu{\toma}^{5/2}
\enorm{\BE-\BE_h}_{\omega,\toma}.
\end{equation}
\end{theorem}

\begin{proof}
We have established in Theorem~\ref{theorem_total_current} that
$\BG_h^{\ba} \in \RT_{\pa+2}(\CTa) \cap \BH_0(\ddiv,\oma)$
with $\div \BG_h^{\ba} = 0$. Then, the existence and uniqueness
of a minimizer $\BH_h^{\ba} \in \ND_{\pa+2}(\CTa) \cap \BH_0(\ccurl,\oma)$
to~\eqref{eq_definition_BH_h_ba} follows from~\cite[Theorem 3.3]{chaumontfrelet_vohralik_2022b}.
The same reasoning also applies to the minimization in an unweighted norm, and since
$\psia \cc^{-1} \curl \BE_h \in \ND_{\pa+1}(\CTa)$,
\cite[Theorem 3.3]{chaumontfrelet_vohralik_2022b} additionally ensures that
\begin{multline}
\label{tmp_unweighted_minimizers}
\min_{\substack{
\bv_h \in \ND_{p+2}(\CTa) \cap \BH_0(\ccurl,\oma)
\\
\curl \bv_h = \BG_h^{\ba}
}}
\|\psia \cc \curl \BE_h-i\omega\bv_h\|_{\oma}
\\
\lesssim_{\CTa}
\min_{\substack{
\bv \in \BH_0(\ccurl,\oma)
\\
\curl \bv = \BG_h^{\ba}
}}
\|\psia \cc \curl \BE_h-i\omega\bv\|_{\oma}.
\end{multline}
For the sake of shortness, we introduce the notation $\bv_h^\star$
and $\bv^\star$ for the minimizers of both sides of~\eqref{tmp_unweighted_minimizers}.
Then, using that $\|\bw\|_{\cc,\oma} = \|\cc \bw\|_{\cc^{-1},\oma}$ whenever $\bw \in \BL^2(\oma)$,
we further infer that
\begin{multline*}
\|\psia \curl \BE_h-i\omega\cc^{-1}\BH_h^{\ba}\|_{\cc,\oma}
\leq
\|\psia \cc \curl \BE_h-i\omega \bv_h^\star\|_{\cc^{-1},\oma}
\leq
\frac{1}{\sqrt{\chimina}}
\|\psia \cc \curl \BE_h-i\omega \bv_h^\star\|_{\oma}
\\
\lesssim_{\CTa}
\frac{1}{\sqrt{\chimina}}
\|\psia \cc \curl \BE_h-i\omega \bv^\star\|_{\oma}
\leq
\frac{1}{\sqrt{\chimina}}
\|\psia \cc \curl \BE_h-i\omega \BH^{\ba}\|_{\oma}
\\
\leq
\sqrt{\contrastmu{\oma}}
\|\psia \cc \curl \BE_h-i\omega \BH^{\ba}\|_{\cc^{-1},\oma}
=
\sqrt{\contrastmu{\oma}}
\|\psia \curl \BE_h-i\omega \cc^{-1} \BH^{\ba}\|_{\cc,\oma}.
\end{multline*}
We finally observe that by combining~\eqref{eq_estimate_BH_ba}
and~\eqref{eq_estimate_BG_h_ba}, we obtain
\begin{align*}
&
\sqrt{\contrastmu{\oma}}
\|\psia \curl \BE_h-i\omega \cc^{-1} \BH^{\ba}\|_{\cc,\oma}.
\\
&\lesssim_{\tCTa}
\sqrt{\contrastmu{\oma}}
\left (
\|\curl (\BE-\BE_h)\|_{\cc,\oma}
+
\sqrt{\contrastmu{\toma}} \ha\|\BG^{\ba}-\BG_h^{\ba}\|_{\mm,\oma}
\right )
\\
&\lesssim_{\tCTa}
\sqrt{\contrastmu{\oma}}
\left (
\|\curl (\BE-\BE_h)\|_{\cc,\oma}
+
\contrastmu{\toma}^2
\contrasteps{\toma}
\left (1 + \frac{\omega\ha}{\cminD{\toma}}\right )
\enorm{\BE-\BE_h}_{\omega,\toma}
\right )
\\
&\lesssim_{\tCTa}
\sqrt{\contrastmu{\oma}}
\left (
1 + \contrastmu{\toma}^2\contrasteps{\toma}
\left (1 + \frac{\omega\ha}{\cminD{\toma}}\right )
\right )
\enorm{\BE-\BE_h}_{\omega,\toma},
\end{align*}
and the result follows since $\contrastmu{\oma} \leq \contrastmu{\toma}$ and
\begin{equation*}
1 \leq
\contrastmu{\toma}^2\contrasteps{\toma}
\left (1 + \frac{\omega\ha}{\cminD{\toma}}\right ).
\end{equation*}
\end{proof}

\begin{remark}[Image of the curl]
Theorem~\ref{theorem_magnetic_reconstruction} shows that $\BH_h^{\ba}$ is well-defined
for each $\ba$, and therefore, after summation, that
\begin{equation*}
i\omega \curl \BH_h = i\omega \BJ_h + \omega^2 \BD_h.
\end{equation*}
Hence, $i\omega \BJ_h + \omega^2 \BD_h$ is in the image of the curl operator.
This is, in general, a stronger property than the observation that
\begin{equation*}
-\omega^2 \div \BD_h = i\omega \div \BJ_h
\end{equation*}
which naturally follows from the local divergence constraints defining $\BD_h^{\ba}$.
This is to be linked with~\cite{chaumontfrelet_2024a} where a similar observation
is made for the magnetostatic problem, but with a completely different proof.
\end{remark}

\subsection{Efficiency}

We can now conclude the efficiency proof by combining the results established in this section.

\begin{proof}[Proof of Theorem~\ref{theorem_efficiency}]
Recalling the definition of $\eta_{\ddiv,K}$ in~\eqref{eq_local_estimators},
the fact that the hat functions $\psia$ form a partition of unity~\eqref{eq_partition_unity},
and that $\BD_h$ decomposes into local contributions $\BD_h^{\ba}$
according to~\eqref{eq_definition_BD_h}, we have
\begin{multline*}
\eta_{\ddiv,K}
=
\omega \|\BE_h-\ee^{-1} \BD_h\|_{\ee,K}
=
\omega \left \|\sum_{\ba \in \CV(K)} (\psia \BE_h-\ee^{-1} \BD_h^{\ba})\right \|_{\ee,K}
\\
\leq
\omega \sum_{\ba \in \CV(K)}\|\psia\BE_h-\ee^{-1} \BD_h^{\ba}\|_{\ee,K}
\leq
\omega \sum_{\ba \in \CV(K)}\|\psia\BE_h-\ee^{-1} \BD_h^{\ba}\|_{\ee,\oma}.
\end{multline*}
Then, recalling~\eqref{eq_estimate_BD_h_ba}, we conclude that
\begin{multline}
\label{tmp_lowerbound_eta_div}
\eta_{\ddiv,K}
\lesssim_{\CTK}
\omega \sum_{\ba \in \CV(K)}
\contrasteps{\oma}
\|\BE-\BE_h\|_{\ee,\oma}
\lesssim_{\CTK}
\contrasteps{\omK}
\omega \|\BE-\BE_h\|_{\ee,\omK}.
\\
\lesssim_{\tCTK}
\left (
1 + \frac{\omega h_K}{\cminD{\tomK}}
\right )
\contrasteps{\tomK}
\contrastmu{\tomK}^{5/2}
\enorm{\BE-\BE_h}_{\omega,\tomK},
\end{multline}
Following a similar path for $\eta_{\ccurl,K}$, but using
\eqref{eq_estimate_BH_h_ba} instead of~\eqref{eq_estimate_BD_h_ba},
we arrive at
\begin{equation}
\label{tmp_lowerbound_eta_curl}
\eta_{\ccurl,K}
\lesssim_{\tCTK}
\left (
1 + \frac{\omega h_K}{\cminD{\tomK}}
\right )
\contrasteps{\tomK}
\contrastmu{\tomK}^{5/2}
\enorm{\BE-\BE_h}_{\omega,\tomK},
\end{equation}
since $\ha \lesssim_{\CTa} h_K$ due to the shape-regularity assumption on the mesh.
Then,~\eqref{eq_efficiency_K} follows from~\eqref{tmp_lowerbound_eta_div} and
\eqref{tmp_lowerbound_eta_curl}.
\end{proof}

\section{Numerical examples}
\label{section_numerical_examples}

This section illustrates our key theoretical results with a set of numerical examples.

\subsection{Continuous setting}

Throughout this section $\Omega \eq (0,1)^3$ is the unit cube, and the coefficients
$\cc = \ee = \BI$ are constant. We will select frequencies of the form
$\omega = 2\pi(m/2 + \delta)$ with $m \in \mathbb N$ and $\delta > 0$. Notice
that the case $\delta = 0$ corresponds to a resonance frequency.

Letting $\be_2 = (0,1,0)$, the right-hand side reads
\begin{equation*}
\BJ(\bx) \eq \sin(m\pi\bx_3)\be_2,
\end{equation*}
and the associated solution is
\begin{equation*}
\BE(\bx) \eq \frac{1}{k^2}
\left (
(\cos(k\bx_1)-1)-(\cos(k)-1)\frac{\sin(k\bx_1)}{\sin(k)}
\right )
\sin(m\pi\bx_3)\be_2
\end{equation*}
where $k \eq \sqrt{\omega^2 - (m\pi)^2}$.

\subsection{Discrete setting}

The domain $\Omega$ is partitioned into unstructured tetrahedral meshes $\CT_h$.
We employ {\tt mmg3d} to generate such meshes~\cite{dobrzynski_2012a}. The parameter
$h$ corresponds to the argument {\tt -hmax} passed to {\tt mmg3d} when generating each mesh.
The mesh sizes go through zero through the sequence $h = 1,0.5,0.25,0.125$.
On these meshes, we consider N\'ed\'elec finite element of the first family, i.e.
$\BW_h = \ND_p(\CT_h) \cap \BH_0(\ccurl,\Omega)$, with $p$ ranging from $1$ to $3$.

\subsection{Results}

Our theoretical results indicate that we may expect the estimator to be unreliable on coarse
meshes when $m$ is large or $\delta$ is close to zero. This lack of reliability should be only
pre-asymptotic, and disappear on fine meshes or when using higher polynomial degrees. We also
expect higher polynomial degrees to be more robust.

To lighten notation in the graphs below, we set
\begin{equation*}
\CE \eq \enorm{\BE-\BE_h}_{\omega,\Omega}
\qquad
\CI \eq \frac{\eta}{\CE}.
\end{equation*}
As discussed above, we expect that $\CI \leq 1$ for coarse meshes, and $\CI \to 1$
as $(\omega h)/(\cminD{\Omega} p) \to 0$.

We first set $m=3$ and let $\delta = 10^{-2},10^{-3},10^{-4}$ approach zero.
Figures~\ref{figure_error_m3_p1},~\ref{figure_error_m3_p2} and~\ref{figure_error_m3_p3} represent
the behaviour error $\CE$ and the estimator $\eta$ as the mesh is refined for the different
choices of $p$. The expected convergence rates are observed as well as a the asymptotic guaranteed
reliability of the estimator. This is seen in more details on Figure~\ref{figure_effectivity_m3}.
In particular, we indeed see that, as predicted, the error underestimation is more pronounced for
smaller $\delta$, but reduced when using a higher value of $p$.

Figures~\ref{figure_error_m5_p1},~\ref{figure_error_m5_p2},~\ref{figure_error_m5_p3} and
\ref{figure_effectivity_m5} present the same results for the case $m=5$. The comments made
above for the case $m=3$ still applies. We further see that the error estimation is more
pronounced for $m=5$ that $m=3$, which is in line with the theoretical prediction.

\begin{subfigures}
\begin{figure}
\begin{minipage}{\figsize\linewidth}
\begin{tikzpicture}
\begin{axis}
[
	xmode = log,
	ymode = log,
	x dir = reverse,
	xlabel = {$h$ ($\delta = 10^{-2}$)},
	width=\linewidth,
	xtick={1.000,0.500,0.250,0.125},
	xticklabels={$2^0$,$2^{-1}$,$2^{-2}$,$2^{-3}$},
	ymin = 1.e-2,
	ymax = 2
]

\plot table[x=h,y=err] {figures/data/M3/curve_F1.51_P1.txt}
node[pos=0.4,pin=45:{$\CE$}] {};
\plot table[x=h,y=est] {figures/data/M3/curve_F1.51_P1.txt}
node[pos=0.1,pin=-90:{$\eta$}] {};

\plot[domain=0.4:0.125,dashed] {x^(2.)};
\SlopeTriangle{.7}{-.15}{.1}{-2}{$h^2$}{}

\end{axis}
\end{tikzpicture}
\end{minipage}
\begin{minipage}{\figsize\linewidth}
\begin{tikzpicture}
\begin{axis}
[
	xmode = log,
	ymode = log,
	x dir = reverse,
	xlabel = {$h$ ($\delta = 10^{-3}$)},
	width=\linewidth,
	xtick={1.000,0.500,0.250,0.125},
	xticklabels={$2^0$,$2^{-1}$,$2^{-2}$,$2^{-3}$},
	ymin = 1.e-2,
	ymax = 2
]

\plot table[x=h,y=err] {figures/data/M3/curve_F1.501_P1.txt}
node[pos=0.8,pin= 90:{$\CE$}] {};
\plot table[x=h,y=est] {figures/data/M3/curve_F1.501_P1.txt}
node[pos=0.1,pin=-90:{$\eta$}] {};

\plot[domain=0.4:0.125,dashed] {x^(2.)};
\SlopeTriangle{.7}{-.15}{.1}{-2}{$h^2$}{}

\end{axis}
\end{tikzpicture}
\end{minipage}

\begin{minipage}{\figsize\linewidth}
\begin{tikzpicture}
\begin{axis}
[
	xmode = log,
	ymode = log,
	x dir = reverse,
	xlabel = {$h$ ($\delta = 10^{-4}$)},
	width=\linewidth,
	xtick={1.000,0.500,0.250,0.125},
	xticklabels={$2^0$,$2^{-1}$,$2^{-2}$,$2^{-3}$},
	ymin = 1.e-2,
	ymax = 2
]

\plot table[x=h,y=err] {figures/data/M3/curve_F1.5001_P1.txt}
node[pos=0.9,pin= 90:{$\CE$}] {};
\plot table[x=h,y=est] {figures/data/M3/curve_F1.5001_P1.txt}
node[pos=0.1,pin=-90:{$\eta$}] {};

\plot[domain=0.4:0.125,dashed] {x^(2.)};
\SlopeTriangle{.7}{-.15}{.1}{-2}{$h^2$}{}

\end{axis}
\end{tikzpicture}
\end{minipage}
\begin{minipage}{\figsize\linewidth}
\caption{$m=3$ and $p=1$}
\label{figure_error_m3_p1}
\end{minipage}
\end{figure}


\begin{figure}
\begin{minipage}{\figsize\linewidth}
\begin{tikzpicture}
\begin{axis}
[
	xmode = log,
	ymode = log,
	x dir = reverse,
	xlabel = {$h$ ($\delta = 10^{-2}$)},
	width=\linewidth,
	xtick={1.000,0.500,0.250,0.125},
	xticklabels={$2^0$,$2^{-1}$,$2^{-2}$,$2^{-3}$},
	ymin = 1.e-3,
	ymax = 4
]

\plot table[x=h,y=err] {figures/data/M3/curve_F1.51_P2.txt}
node[pos=0.3,pin= 45:{$\CE$}] {};
\plot table[x=h,y=est] {figures/data/M3/curve_F1.51_P2.txt}
node[pos=0.1,pin=-90:{$\eta$}] {};

\plot[domain=0.4:0.125,dashed] {0.55*x^(3.)};
\SlopeTriangle{.6}{-.15}{.1}{-3}{$h^3$}{}

\end{axis}
\end{tikzpicture}
\end{minipage}
\begin{minipage}{\figsize\linewidth}
\begin{tikzpicture}
\begin{axis}
[
	xmode = log,
	ymode = log,
	x dir = reverse,
	xlabel = {$h$ ($\delta = 10^{-3}$)},
	width=\linewidth,
	xtick={1.000,0.500,0.250,0.125},
	xticklabels={$2^0$,$2^{-1}$,$2^{-2}$,$2^{-3}$},
	ymin = 1.e-3,
	ymax = 4
]

\plot table[x=h,y=err] {figures/data/M3/curve_F1.501_P2.txt}
node[pos=0.3,pin= 45:{$\CE$}] {};
\plot table[x=h,y=est] {figures/data/M3/curve_F1.501_P2.txt}
node[pos=0.1,pin=-90:{$\eta$}] {};

\plot[domain=0.4:0.125,dashed] {0.55*x^(3.)};
\SlopeTriangle{.6}{-.15}{.1}{-3}{$h^3$}{}

\end{axis}
\end{tikzpicture}
\end{minipage}

\begin{minipage}{\figsize\linewidth}
\begin{tikzpicture}
\begin{axis}
[
	xmode = log,
	ymode = log,
	x dir = reverse,
	xlabel = {$h$ ($\delta = 10^{-4}$)},
	width=\linewidth,
	xtick={1.000,0.500,0.250,0.125},
	xticklabels={$2^0$,$2^{-1}$,$2^{-2}$,$2^{-3}$},
	ymin = 1.e-3,
	ymax = 4
]

\plot table[x=h,y=err] {figures/data/M3/curve_F1.5001_P2.txt}
node[pos=0.8,pin= 90:{$\CE$}] {};
\plot table[x=h,y=est] {figures/data/M3/curve_F1.5001_P2.txt}
node[pos=0.1,pin=-90:{$\eta$}] {};

\plot[domain=0.4:0.125,dashed] {0.55*x^(3.)};
\SlopeTriangle{.6}{-.15}{.1}{-3}{$h^3$}{}

\end{axis}
\end{tikzpicture}
\end{minipage}
\begin{minipage}{\figsize\linewidth}
\caption{$m=3$ and $p=2$}
\label{figure_error_m3_p2}
\end{minipage}
\end{figure}


\begin{figure}
\begin{minipage}{\figsize\linewidth}
\begin{tikzpicture}
\begin{axis}
[
	xmode = log,
	ymode = log,
	x dir = reverse,
	xlabel = {$h$ ($\delta = 10^{-2}$)},
	width=\linewidth,
	xtick={1.000,0.500,0.250,0.125},
	xticklabels={$2^0$,$2^{-1}$,$2^{-2}$,$2^{-3}$},
	ymin = 5.e-5,
	ymax = 3
]

\plot table[x=h,y=err] {figures/data/M3/curve_F1.51_P3.txt}
node[pos=0.2,pin= 45:{$\CE$}] {};
\plot table[x=h,y=est] {figures/data/M3/curve_F1.51_P3.txt}
node[pos=0.1,pin=-90:{$\eta$}] {};

\plot[domain=0.4:0.125,dashed] {0.25*x^(4.)};
\SlopeTriangle{.6}{-.15}{.1}{-4}{$h^4$}{}

\end{axis}
\end{tikzpicture}
\end{minipage}
\begin{minipage}{\figsize\linewidth}
\begin{tikzpicture}
\begin{axis}
[
	xmode = log,
	ymode = log,
	x dir = reverse,
	xlabel = {$h$ ($\delta = 10^{-3}$)},
	width=\linewidth,
	xtick={1.000,0.500,0.250,0.125},
	xticklabels={$2^0$,$2^{-1}$,$2^{-2}$,$2^{-3}$},
	ymin = 5.e-5,
	ymax = 3
]

\plot table[x=h,y=err] {figures/data/M3/curve_F1.501_P3.txt}
node[pos=0.3,pin= 45:{$\CE$}] {};
\plot table[x=h,y=est] {figures/data/M3/curve_F1.501_P3.txt}
node[pos=0.1,pin=-90:{$\eta$}] {};

\plot[domain=0.4:0.125,dashed] {0.25*x^(4.)};
\SlopeTriangle{.6}{-.15}{.1}{-4}{$h^4$}{}

\end{axis}
\end{tikzpicture}
\end{minipage}

\begin{minipage}{\figsize\linewidth}
\begin{tikzpicture}
\begin{axis}
[
	xmode = log,
	ymode = log,
	x dir = reverse,
	xlabel = {$h$ ($\delta = 10^{-4}$)},
	width=\linewidth,
	xtick={1.000,0.500,0.250,0.125},
	xticklabels={$2^0$,$2^{-1}$,$2^{-2}$,$2^{-3}$},
	ymin = 5.e-5,
	ymax = 3
]

\plot table[x=h,y=err] {figures/data/M3/curve_F1.5001_P3.txt}
node[pos=0.5,pin= 45:{$\CE$}] {};
\plot table[x=h,y=est] {figures/data/M3/curve_F1.5001_P3.txt}
node[pos=0.1,pin=-90:{$\eta$}] {};

\plot[domain=0.4:0.125,dashed] {0.25*x^(4.)};
\SlopeTriangle{.6}{-.15}{.1}{-4}{$h^4$}{}

\end{axis}
\end{tikzpicture}
\end{minipage}
\begin{minipage}{\figsize\linewidth}
\caption{$m=3$ and $p=3$}
\label{figure_error_m3_p3}
\end{minipage}
\end{figure}
\end{subfigures}

\begin{figure}
\begin{minipage}{\figsize\linewidth}
\begin{tikzpicture}
\begin{axis}
[
	xmode = log,
	x dir = reverse,
	xlabel = {$h$ ($\delta = 10^{-2}$)},
	width=\linewidth,
	xtick={1.000,0.500,0.250,0.125},
	xticklabels={$2^0$,$2^{-1}$,$2^{-2}$,$2^{-3}$},
	ymin = 0,
	ymax = 1.1
]

\plot table[x=h,y expr=\thisrow{est}/\thisrow{err}] {figures/data/M3/curve_F1.51_P1.txt}
node[pos=0.90,pin= -90:{$p=1$}] {};
\plot table[x=h,y expr=\thisrow{est}/\thisrow{err}] {figures/data/M3/curve_F1.51_P2.txt}
node[pos=0.35,pin=- 30:{$p=2$}] {};
\plot table[x=h,y expr=\thisrow{est}/\thisrow{err}] {figures/data/M3/curve_F1.51_P3.txt}
node[pos=0.40,pin= 139:{$p=3$}] {};

\end{axis}
\end{tikzpicture}
\end{minipage}
\begin{minipage}{\figsize\linewidth}
\begin{tikzpicture}
\begin{axis}
[
	xmode = log,
	x dir = reverse,
	xlabel = {$h$ ($\delta = 10^{-3}$)},
	width=\linewidth,
	xtick={1.000,0.500,0.250,0.125},
	xticklabels={$2^0$,$2^{-1}$,$2^{-2}$,$2^{-3}$},
	ymin = 0,
	ymax = 1.1
]

\plot table[x=h,y expr=\thisrow{est}/\thisrow{err}] {figures/data/M3/curve_F1.501_P1.txt}
node[pos=0.95,pin= -90:{$p=1$}] {};
\plot table[x=h,y expr=\thisrow{est}/\thisrow{err}] {figures/data/M3/curve_F1.501_P2.txt}
node[pos=0.45,pin={[pin distance=1.2cm]-30:{$p=2$}}] {};
\plot table[x=h,y expr=\thisrow{est}/\thisrow{err}] {figures/data/M3/curve_F1.501_P3.txt}
node[pos=0.70,pin=180:{$p=3$}] {};

\end{axis}
\end{tikzpicture}
\end{minipage}

\begin{minipage}{\figsize\linewidth}
\begin{tikzpicture}
\begin{axis}
[
	xmode = log,
	x dir = reverse,
	xlabel = {$h$ ($\delta = 10^{-4}$)},
	width=\linewidth,
	xtick={1.000,0.500,0.250,0.125},
	xticklabels={$2^0$,$2^{-1}$,$2^{-2}$,$2^{-3}$},
	ymin = 0,
	ymax = 1.1
]

\plot table[x=h,y expr=\thisrow{est}/\thisrow{err}] {figures/data/M3/curve_F1.5001_P1.txt}
node[pos=0.1,pin={$p=1$}] {};
\plot table[x=h,y expr=\thisrow{est}/\thisrow{err}] {figures/data/M3/curve_F1.5001_P2.txt}
node[pos=0.8,pin= 120:{$p=2$}] {};
\plot table[x=h,y expr=\thisrow{est}/\thisrow{err}] {figures/data/M3/curve_F1.5001_P3.txt}
node[pos=0.95,pin= 180:{$p=3$}] {};

\end{axis}
\end{tikzpicture}
\end{minipage}
\begin{minipage}{\figsize\linewidth}
\caption{Effectivity indices $\CI$ for the case $m=3$}
\label{figure_effectivity_m3}
\end{minipage}
\end{figure}

\begin{subfigures}
\begin{figure}
\begin{minipage}{\figsize\linewidth}
\begin{tikzpicture}
\begin{axis}
[
	xmode = log,
	ymode = log,
	x dir = reverse,
	xlabel = {$h$ ($\delta = 10^{-2}$)},
	width=\linewidth,
	xtick={1.000,0.500,0.250,0.125},
	xticklabels={$2^0$,$2^{-1}$,$2^{-2}$,$2^{-3}$},
	ymin = 8.e-2,
	ymax = 20
]

\plot table[x=h,y=err] {figures/data/M5/curve_F2.51_P1.txt}
node[pos=0.8,pin= 90:{$\CE$}] {};
\plot table[x=h,y=est] {figures/data/M5/curve_F2.51_P1.txt}
node[pos=0.3,pin=-90:{$\eta$}] {};

\end{axis}
\end{tikzpicture}
\end{minipage}
\begin{minipage}{\figsize\linewidth}
\begin{tikzpicture}
\begin{axis}
[
	xmode = log,
	ymode = log,
	x dir = reverse,
	xlabel = {$h$ ($\delta = 10^{-3}$)},
	width=\linewidth,
	xtick={1.000,0.500,0.250,0.125},
	xticklabels={$2^0$,$2^{-1}$,$2^{-2}$,$2^{-3}$},
	ymin = 8.e-2,
	ymax = 20
]

\plot table[x=h,y=err] {figures/data/M5/curve_F2.501_P1.txt}
node[pos=0.8,pin= 90:{$\CE$}] {};
\plot table[x=h,y=est] {figures/data/M5/curve_F2.501_P1.txt}
node[pos=0.1,pin=-90:{$\eta$}] {};

\end{axis}
\end{tikzpicture}
\end{minipage}

\begin{minipage}{\figsize\linewidth}
\begin{tikzpicture}
\begin{axis}
[
	xmode = log,
	ymode = log,
	x dir = reverse,
	xlabel = {$h$ ($\delta = 10^{-4}$)},
	width=\linewidth,
	xtick={1.000,0.500,0.250,0.125},
	xticklabels={$2^0$,$2^{-1}$,$2^{-2}$,$2^{-3}$},
	ymin = 8.e-2,
	ymax = 20
]

\plot table[x=h,y=err] {figures/data/M5/curve_F2.5001_P1.txt}
node[pos=0.9,pin= 90:{$\CE$}] {};
\plot table[x=h,y=est] {figures/data/M5/curve_F2.5001_P1.txt}
node[pos=0.1,pin=-90:{$\eta$}] {};

\end{axis}
\end{tikzpicture}
\end{minipage}
\begin{minipage}{\figsize\linewidth}
\caption{$m=5$ and $p=1$}
\label{figure_error_m5_p1}
\end{minipage}
\end{figure}


\begin{figure}
\begin{minipage}{\figsize\linewidth}
\begin{tikzpicture}
\begin{axis}
[
	xmode = log,
	ymode = log,
	x dir = reverse,
	xlabel = {$h$ ($\delta = 10^{-2}$)},
	width=\linewidth,
	xtick={1.000,0.500,0.250,0.125},
	xticklabels={$2^0$,$2^{-1}$,$2^{-2}$,$2^{-3}$},
	ymin = 1.e-2,
	ymax = 3
]

\plot table[x=h,y=err] {figures/data/M5/curve_F2.51_P2.txt}
node[pos=0.4,pin= 45:{$\CE$}] {};
\plot table[x=h,y=est] {figures/data/M5/curve_F2.51_P2.txt}
node[pos=0.1,pin=-90:{$\eta$}] {};

\end{axis}
\end{tikzpicture}
\end{minipage}
\begin{minipage}{\figsize\linewidth}
\begin{tikzpicture}
\begin{axis}
[
	xmode = log,
	ymode = log,
	x dir = reverse,
	xlabel = {$h$ ($\delta = 10^{-3}$)},
	width=\linewidth,
	xtick={1.000,0.500,0.250,0.125},
	xticklabels={$2^0$,$2^{-1}$,$2^{-2}$,$2^{-3}$},
	ymin = 1.e-2,
	ymax = 3
]

\plot table[x=h,y=err] {figures/data/M5/curve_F2.501_P2.txt}
node[pos=0.4,pin= 45:{$\CE$}] {};
\plot table[x=h,y=est] {figures/data/M5/curve_F2.501_P2.txt}
node[pos=0.1,pin=-90:{$\eta$}] {};

\end{axis}
\end{tikzpicture}
\end{minipage}

\begin{minipage}{\figsize\linewidth}
\begin{tikzpicture}
\begin{axis}
[
	xmode = log,
	ymode = log,
	x dir = reverse,
	xlabel = {$h$ ($\delta = 10^{-4}$)},
	width=\linewidth,
	xtick={1.000,0.500,0.250,0.125},
	xticklabels={$2^0$,$2^{-1}$,$2^{-2}$,$2^{-3}$},
	ymin = 1.e-2,
	ymax = 3
]

\plot table[x=h,y=err] {figures/data/M5/curve_F2.5001_P2.txt}
node[pos=0.8,pin= 90:{$\CE$}] {};
\plot table[x=h,y=est] {figures/data/M5/curve_F2.5001_P2.txt}
node[pos=0.1,pin=-90:{$\eta$}] {};

\end{axis}
\end{tikzpicture}
\end{minipage}
\begin{minipage}{\figsize\linewidth}
\caption{$m=5$ and $p=2$}
\label{figure_error_m5_p2}
\end{minipage}
\end{figure}


\begin{figure}
\begin{minipage}{\figsize\linewidth}
\begin{tikzpicture}
\begin{axis}
[
	xmode = log,
	ymode = log,
	x dir = reverse,
	xlabel = {$h$ ($\delta = 10^{-2}$)},
	width=\linewidth,
	xtick={1.000,0.500,0.250,0.125},
	xticklabels={$2^0$,$2^{-1}$,$2^{-2}$,$2^{-3}$},
	ymin = 1.e-3,
	ymax = 10
]

\plot table[x=h,y=err] {figures/data/M5/curve_F2.51_P3.txt}
node[pos=0.4,pin= 45:{$\CE$}] {};
\plot table[x=h,y=est] {figures/data/M5/curve_F2.51_P3.txt}
node[pos=0.1,pin=-90:{$\eta$}] {};

\end{axis}
\end{tikzpicture}
\end{minipage}
\begin{minipage}{\figsize\linewidth}
\begin{tikzpicture}
\begin{axis}
[
	xmode = log,
	ymode = log,
	x dir = reverse,
	xlabel = {$h$ ($\delta = 10^{-3}$)},
	width=\linewidth,
	xtick={1.000,0.500,0.250,0.125},
	xticklabels={$2^0$,$2^{-1}$,$2^{-2}$,$2^{-3}$},
	ymin = 1.e-3,
	ymax = 10
]

\plot table[x=h,y=err] {figures/data/M5/curve_F2.501_P3.txt}
node[pos=0.3,pin= 45:{$\CE$}] {};
\plot table[x=h,y=est] {figures/data/M5/curve_F2.501_P3.txt}
node[pos=0.1,pin=-90:{$\eta$}] {};

\end{axis}
\end{tikzpicture}
\end{minipage}

\begin{minipage}{\figsize\linewidth}
\begin{tikzpicture}
\begin{axis}
[
	xmode = log,
	ymode = log,
	x dir = reverse,
	xlabel = {$h$ ($\delta = 10^{-4}$)},
	width=\linewidth,
	xtick={1.000,0.500,0.250,0.125},
	xticklabels={$2^0$,$2^{-1}$,$2^{-2}$,$2^{-3}$},
	ymin = 1.e-3,
	ymax = 10
]

\plot table[x=h,y=err] {figures/data/M5/curve_F2.5001_P3.txt}
node[pos=0.5,pin= 45:{$\CE$}] {};
\plot table[x=h,y=est] {figures/data/M5/curve_F2.5001_P3.txt}
node[pos=0.1,pin=-90:{$\eta$}] {};

\end{axis}
\end{tikzpicture}
\end{minipage}
\begin{minipage}{\figsize\linewidth}
\caption{$m=5$ and $p=3$}
\label{figure_error_m5_p3}
\end{minipage}
\end{figure}
\end{subfigures}

\begin{figure}
\begin{minipage}{\figsize\linewidth}
\begin{tikzpicture}
\begin{axis}
[
	xmode = log,
	x dir = reverse,
	xlabel = {$h$ ($\delta = 10^{-2}$)},
	width=\linewidth,
	xtick={1.000,0.500,0.250,0.125},
	xticklabels={$2^0$,$2^{-1}$,$2^{-2}$,$2^{-3}$},
	ymin = 0,
	ymax = 1.1
]

\plot table[x=h,y expr=\thisrow{est}/\thisrow{err}] {figures/data/M5/curve_F2.51_P1.txt}
node[pos=0.90,pin= -90:{$p=1$}] {};
\plot table[x=h,y expr=\thisrow{est}/\thisrow{err}] {figures/data/M5/curve_F2.51_P2.txt}
node[pos=0.10,pin=  90:{$p=2$}] {};
\plot table[x=h,y expr=\thisrow{est}/\thisrow{err}] {figures/data/M5/curve_F2.51_P3.txt}
node[pos=0.60,pin= 130:{$p=3$}] {};

\end{axis}
\end{tikzpicture}
\end{minipage}
\begin{minipage}{\figsize\linewidth}
\begin{tikzpicture}
\begin{axis}
[
	xmode = log,
	x dir = reverse,
	xlabel = {$h$ ($\delta = 10^{-3}$)},
	width=\linewidth,
	xtick={1.000,0.500,0.250,0.125},
	xticklabels={$2^0$,$2^{-1}$,$2^{-2}$,$2^{-3}$},
	ymin = 0,
	ymax = 1.1
]

\plot table[x=h,y expr=\thisrow{est}/\thisrow{err}] {figures/data/M5/curve_F2.501_P1.txt}
node[pos=0.95,pin= -90:{$p=1$}] {};
\plot table[x=h,y expr=\thisrow{est}/\thisrow{err}] {figures/data/M5/curve_F2.501_P2.txt}
node[pos=0.75,pin={[pin distance=1.2cm]160:{$p=2$}}] {};
\plot table[x=h,y expr=\thisrow{est}/\thisrow{err}] {figures/data/M5/curve_F2.501_P3.txt}
node[pos=0.95,pin=180:{$p=3$}] {};

\end{axis}
\end{tikzpicture}
\end{minipage}

\begin{minipage}{\figsize\linewidth}
\begin{tikzpicture}
\begin{axis}
[
	xmode = log,
	x dir = reverse,
	xlabel = {$h$ ($\delta = 10^{-4}$)},
	width=\linewidth,
	xtick={1.000,0.500,0.250,0.125},
	xticklabels={$2^0$,$2^{-1}$,$2^{-2}$,$2^{-3}$},
	ymin = 0,
	ymax = 1.1
]

\plot table[x=h,y expr=\thisrow{est}/\thisrow{err}] {figures/data/M5/curve_F2.5001_P1.txt}
node[pos=0.1,pin={$p=1$}] {};
\plot table[x=h,y expr=\thisrow{est}/\thisrow{err}] {figures/data/M5/curve_F2.5001_P2.txt}
node[pos=0.7,pin= 120:{$p=2$}] {};
\plot table[x=h,y expr=\thisrow{est}/\thisrow{err}] {figures/data/M5/curve_F2.5001_P3.txt}
node[pos=0.95,pin= 180:{$p=3$}] {};

\end{axis}
\end{tikzpicture}
\end{minipage}
\begin{minipage}{\figsize\linewidth}
\caption{Effectivity indices $\CI$ for the case $m=5$}
\label{figure_effectivity_m5}
\end{minipage}
\end{figure}

\bibliographystyle{amsplain}
\bibliography{files/bibliography.bib}

\end{document}